\documentclass[reqno,12pt]{amsart}

\usepackage{graphicx,pdfsync,amssymb, latexsym,amsfonts,amsbsy,color,subcaption,esint,mathtools,enumerate,nicefrac,bm}
\usepackage{hyperref}
\usepackage{booktabs}
\usepackage{multirow}
\usepackage{appendix}
\hypersetup{
colorlinks=true,
linkcolor=blue,
filecolor=blue,
urlcolor=blue,
citecolor=blue,
}

\usepackage{float}
\usepackage{mathrsfs}
\restylefloat{table}

\usepackage{enumitem}
\setitemize{itemsep=0.4em}
\setenumerate{itemsep=0.5em}

\usepackage[utf8]{inputenc}% for special characters
\usepackage[T1]{fontenc}% for special characters
\usepackage[top=1.1in, bottom=1in, left=1in, right=1in]{geometry}
\DeclareMathOperator{\curl}{curl}

\newtheorem{theorem}{Theorem}[section]
\newtheorem{lemma}[theorem]{Lemma}
\newtheorem{proposition}[theorem]{Proposition}
\newtheorem{definition}[theorem]{Definition}

\newtheorem{remark}[theorem]{Remark}

\newcommand{\thistheoremname}{}
\newtheorem{genericthm}[theorem]{\thistheoremname}

 \newtheorem*{genericthm*}{\thistheoremname}
\newenvironment{namedthm*}[1]
  {\renewcommand{\thistheoremname}{#1}%
   \begin{genericthm*}}
  {\end{genericthm*}}

\usepackage{times}
\usepackage{setspace}
\newcommand{\dd}{\mathop{}\!\mathrm{d}}
\let\del\partial
\newcommand{\weakstar}{\overset{\ast}{\rightharpoonup}} % 弱*收敛符号
\newcommand{\supp}{\text{supp}\,}

\newcommand{\RR}{\mathring R}
\newcommand{\tRR}{\mathring {\widetilde{R}}}
\newcommand{\R}{\mathbb R}

\newcommand{\ZZ}{\mathbb Z}
\newcommand{\NN}{\mathbb N}
\newcommand{\BMO}{{\rm BMO}}

\newcommand{\ootimes}{\mathbin{\mathring{\otimes}}}

\newcommand{\uin}[0]{u^{\textup{in}}}

\newcommand{\wcq}[1][q+1]{w^{\textup{(c)}}_{#1}}

\newcommand{\wttq}[1][q+1]{w^{\textup{(t)}}_{#1}}
\newcommand{\wttqc}[1][q+1]{w^{\textup{(tc)}}_{#1}}

\newcommand{\PP}[0]{\mathbb{P}_{\neq 0}}

\newcommand{\wqloc}{w^\textup{loc}_{q+1}}

\newcommand{\TT}[0]{\mathsf{T}}
\newcommand{\TTT}[0]{\mathbb{T}}

\newcommand{\wtq}[1][q+1]{w^{\textup{(ns)}}_{2q}}
\newcommand{\wpq}[1][q+1]{w^{\textup{(pri)}}_{2q}}
\newcommand{\wpp}{ {w}^{\text{(p)}} }
\newcommand{\wpm}{ {w}^{\text{(p),main}} }
\newcommand{\wpr}{ {w}^{\text{(p),rem}} }
\newcommand{\wss}{ {w}^{\text{(s)}} }

\newcommand{\ii}{\textup{i}}

\newcommand{\ve}{v^{\epsilon}}

\newcommand{\Div}{\text{div}}

\newcommand{\vq}{v_q}
\newcommand{\vlql}{v^{\text{loc}}_{\ell_q}}
\newcommand{\vlqnl}{v^{\text{non-loc}}_{\ell_q}}

\newcommand{\vdq}{v^{\textup{D}}_q}
\newcommand{\Rem}{R^{\text{rem}}_q}
\newcommand{\LinfB}[1][q+1]{\widetilde{L}^\infty_T B^{\frac{1}{2}}_{2,1}}
\newcommand{\LoB}[1][q+1]{{L}^1_t B^{1/2}_{2,1}}
\newcommand{\LoBt}[1][q+1]{\widetilde {L}^1_t B^{5/2}_{2,1}}
\newcommand{\spt}{\text{spt}}

\newcommand{\vql}{v^{\text{loc}}_q}
\newcommand{\vqnl}{v^{\text{non-loc}}_q}
\allowdisplaybreaks[4]
\numberwithin{equation}{section}

\begin{document}

\title{Non-uniqueness of smooth solutions to the Navier-Stokes
equations on torus $\TTT^2$}

\author{Changxing Miao}
\address[Changxing Miao]{Institute  of Applied Physics and Computational Mathematics, Beijing, China.}

\email{miao\_changxing@iapcm.ac.cn}

\author{Yao Nie}

\address[Yao Nie]{School of Mathematical Sciences and LPMC, Nankai University, Tianjin, China.}

 \email{nieyao@nankai.edu.cn}

\author{Weikui Ye}

\address[Weikui Ye]{School of Mathematical Sciences, South China Normal University, Guangzhou,  China}

 \email{904817751@qq.com}

%    General info
%\subjclass[2020]{76D03, 35Q35}

\date{\today}
\maketitle

\begin{abstract}The local well-posedness theory for the incompressible Navier-Stokes equations in $\BMO^{-1}$ has attracted considerable attention over the past two decades. In a recent breakthrough, Coiculescu and Palasek (Invent. Math., 2025) settled the three-dimensional case by demonstrating the existence of two distinct global solutions, both smooth for $t>0$, evolving from a common initial datum in ${\rm BMO}^{-1}(\mathbb{T}^3)$. However, the two-dimensional case remains open. In this paper, we solve the two-dimensional problem. Unlike its three-dimensional counterpart, the two-dimensional setting presents additional difficulties stemming from the geometric intersections of two-dimensional Mikado flows. To overcome these difficulties, we develop a heat-dominated Fourier mode flow built upon steady two-dimensional Euler flows, and present  the proof using a new iterative scheme.
\end{abstract}



\section{Introduction}
In this paper, we consider the incompressible Navier-Stokes equations  on  
$\TTT^2$ 
:
\begin{equation}
\left\{
\begin{aligned}
&
\partial_t u-\Delta u+(u\cdot\nabla) u+\nabla p=0,\\
&
\nabla\cdot
 u=0,
\end{aligned}
\right. \tag{NS}\label
{NS}
\end{equation}
supplemented with a divergence-free initial datum 
$\uin $. Here $u$ denotes the velocity  of fluid and $p$ the pressure.

When posed on 
$\R^d$, the equations \eqref{NS} admit a natural scaling invariance: if $(u,p)$ is a solution, then for every $\lambda>0$
 the rescaled pair
\[
u_
\lambda(t,x)=\lambda u(\lambda^2 t,\lambda x),\qquad
p_
\lambda(t,x)=\lambda^2 p(\lambda^2 t,\lambda
 x),
\]
also satisfies 
\eqref{NS}, corresponding to the initial datum $\lambda \uin (\lambda x)$. This scaling invariance motivates the notion of a \emph{critical} space---a Banach space $(X,\|\cdot\|_X)$
 of initial data  such that
\[
\|\uin(\cdot)\|_X\sim\|\lambda \uin(\lambda\cdot)\|_X, \qquad\forall\lambda
>0.
\]
The well-posedness theories in various critical spaces has been extensively investigated. Fujita and Kato~\cite{FK} established local well-posedness for large initial data and global well-posedness for small data in the Sobolev space $\dot H^{\frac{d}{2}-1}(\mathbb{R}^d)$, while Kato \cite{Kato84} treated the Lebesgue space $L^d(\mathbb{R}^d)$. Global well-posedness for small data was further extended to the critical Besov spaces $\dot B_{p,\infty}^{d/p-1}(\R^d)$ ($1\le p<\infty$) by Cannone and  Planchon \cite{Cannone95, Can97, CM95, Planchon96}. Koch and Tataru \cite{KochTataru01} later extended the theory to the space ${\rm BMO}^{-1}(\R^d)$. These spaces are linked by the following continuous embeddings:
\[
\dot H^{\frac{d}{2}-1}\subset L^d\subset\dot B_{p,\infty}^{d/p-1}\subset {\rm BMO}^{-1}\subset\dot F_{\infty,r}^{-1}\subset\dot B_{\infty,\infty
}^{-1},
\qquad(d<p<\infty,\;
r>2).
\]
In the largest critical space $\dot B^{-1}_{\infty,\infty}$, Bourgain and Pavlović~\cite{BourgainPavlovic03} showed the {ill-posedness} via a norm-inflation phenomenon: there exist  initial data $\varphi$  with arbitrarily small $\|\varphi\|_{\dot B^{-1}_{\infty,\infty}}$ such that the corresponding solutions become arbitrarily large after an arbitrarily short time. Yoneda \cite{Yoneda03} extended this ill-posedness to the spaces $\dot F^{-1}_{\infty,r}$ for all $r>2$. As a consequence, ${\rm BMO}^{-1}$ is the largest critical space in which global well-posedness is known to hold for small initial data. Whether local well-posedness remains true for large data in ${\rm BMO}^{-1}$ is then a natural and central open problem.  In the seminal  work \cite{CP}, Coiculescu and Palasek gave a negative answer to this problem in the three-dimensional setting. They constructed initial data in ${\rm BMO}^{-1}(\mathbb{T}^3)$ that admit two distinct global solutions, both smooth for $t>0$.   Whether the same phenomenon occurs in two dimensions is an open problem, as proposed in \cite[Remark 1.7]{CP}. In this paper, we settle this problem and give a definitive answer as  follows:
\begin{theorem}\label{t:main}
The system \eqref{NS} admits two distinct global solutions $u, \widetilde{u} \in C^\infty_{t,x}(\R^{+}\times\TTT^2)$  with the same initial data $\uin$ and
\begin{align*}
u, \widetilde{u}  \in  L^{\infty}([0,\infty);{\rm BMO}^{-1}(\TTT^2))\cap L^{2}([0,\infty);L^r(\TTT^2))\cap L^{1}([0,\infty);W^{1,r}(\TTT^2))
\end{align*}
for all $1\le r<\infty$.
\end{theorem}
\begin{remark}By a modification of the method developed in this paper and combining the iterative scheme introduced in \cite{MNY-ar1, MNY-ar2}, one can expect similar results  on $\R^2$.
\end{remark}

\noindent{\bf Main ideas}. We will employ an iterative framework to prove  Theorem \ref{t:main}. Specifically, assuming that $\{(u_l,p_l)\}_{1\le l\le m-1}$ are solutions to the equations \eqref{NS},  we  construct $u_m = u_{m-2} + w_m$ as a solution to the system \eqref{NS}. Here, $w_m$ consists of three components: $\wpp_m$, $\wss_m$, and $w^{\text{(ns)}}_m$.
\begin{enumerate}
    \item [$\bullet$] \textit{ The heat-dominated Fourier mode flow $\wpp_m$} is constructed  as
    \begin{align*}
     \wpp_{m} \sim  \lambda_{m} e^{-\lambda^2_{m}t} \sum_{k \in \Lambda} a_k\left({\rm Id}+\tfrac{\mathcal{R}\wpp_{m-1}}{1000\|\mathcal{R}\wpp_{m-1}\|_{L^\infty}}\right) \chi_{m-1} \phi(\lambda^{1/4}_{m}k\cdot x) e^{\mathrm{i}\lambda_{m} k \cdot x} \bar{k} +{\rm l.o.t .},
    \end{align*}
   where $\phi(\lambda^{1/4}_{m}k\cdot x) \bar{k}$ is a 2D Mikado flow. These flows are inevitably geometrically intersected, and consequently, a delicate error emerges from the quadratic interaction $\Div(\wpp_{m}\otimes \wpp_{m})$.  To overcome this difficulty, we design the highly oscillatory part of $\wpp_m$ as  $e^{-\lambda^2_m t}\sum_{k\in\Lambda} e^{\ii\lambda_mk\cdot x}$, which is not only a solution of the steady two-dimensional Euler equations but also of the free heat equation. This key property allows the  oscillatory error  to be absorbed into the pressure term (see \eqref{struc}), while keeping the residual errors under control.
        \item [$\bullet$] \textit{The inverse cascade-dominated flow $\wss_m$} is given by 
        \[\wss_{m}=\wpp_{m-1}e^{-2\lambda^2_{m}t},\]
        which is introduced by \cite{CP}. The constructions of $\wpp_{2q}$ and $\wss_{2q}$ guarantees that 
        \begin{align*}
            \partial_t\wss_{m}+\Div(\wpp_{m}\otimes \wpp_{m})=\nabla P+{\rm l. o. t.},
        \end{align*}
        and $\wss_m$ oscillates with a frequency proportional to $\lambda_{m-1}$.
           \item [$\bullet$] \textit{The incompressible perturbation flow $w^{\rm (ns)}_m$} is designed to absorb  all lower-order terms (denoted by $F_m$) via  solving the forced Navier-Stokes equations \eqref{e:wt} with zero initial data. Observing that $F_m$ is small in $ L^1_t B^{-\frac{1}{2}}_{\infty,1}$, we adopt a different framework from the $\BMO^{-1}$ setting used in \cite{CP}. Specifically, we treat system \eqref{e:wt} directly in the subcritical space $\widetilde L^\infty_t B^{-\frac{1}{2}}_{\infty,1} \cap L^1_t B^{\frac{3}{2}}_{\infty,1}$, which shows that $w^{\mathrm{(ns)}}_m$ remains small in $\BMO^{-1}$.          
\end{enumerate}
Based on the construction of  $w_m$, we set
\begin{align*}
 u^{\text{odd}}=\sum_{l=0}^\infty w_{2l+1},\, \,\,u^{\text{even}}=\sum_{l=1}^\infty w_{2l}.
\end{align*}
 The initial data conditions \eqref{ini-wm} and $(\wpp_1, w^{\text{(s)}}_1,w^{\text{(ns)}}_1)=(\lambda_1e^{-\lambda_1^2 t}\sin(\lambda_1 x_1)\,e_2, 0, 0)\equiv0$ give
\begin{equation}
\left\{
\begin{aligned}\notag
u^{\text{odd}}(0,x)=&  \wpp_1(x,0) +\wss_3(x,0)+   \wpp_3(x,0)+\wss_5(x,0)+   \wpp_5(x,0)+\cdots\\
\qquad\qquad&  \quad\parallel\qquad\quad\quad\,  \parallel\quad\qquad \quad\,\,\,  \,\parallel\qquad\qquad\,\,\,\,  \parallel\\
u^{\text{even}}(0,x)=&\wss_2(x,0)+   \wpp_2(x,0) +\wss_4(x,0)+   \wpp_4(x,0)+\cdots
\end{aligned}
\right.
\end{equation}
which implies that $u^{\text{odd}}(0,x)=u^{\text{even}}(0,x)$ in the sense of distributions. 

Thanks to the built-in time factor $e^{-\lambda^2_m t}$ in the construction of $\{\wpp_m,\wss_m\}_{m\ge 2}$, we anticipate that  their sizes are exponentially small ($\sim e^{-\lambda^2_2\lambda^{-2}_1}$) in $B^{-1}_{\infty,\infty}$, whereas $\{w^{\text{(ns)}}_m\}_{m\ge2}$ exhibit algebraic decay $\lambda^{-10}_{m-1}$, at time $\lambda^{-2}_1$. These differences ensure a clear separation between $\wpp_1(\lambda^{-2}_1,x)$ and $\{\wpp_m,\wss_m, w^{\text{(ns)}}_m\}_{m\ge 2}$ in ${\rm BMO}^{-1}$, proving that $u^{\text{odd}}\neq u^{\text{even}}$.\\

%Our non-uniqueness mechanism is motivated by the seminal work of \cite{CP}. 

\noindent{\bf Some non-uniqueness results.} As previously mentioned,  
Coiculescu and Palasek \cite{CP}  have pioneered a novel mechanism to construct two distinct global  solutions in $\BMO^{-1}(\TTT^3)$. The non-uniqueness mechanism builds upon the one proposed by  Palasek in the context of the Obukhov dyadic model. A distinct mechanism  has been developed in \cite{CDP} to produce a solution of the Navier–Stokes equations that exhibits Type I blow-up of the 
$L^\infty$ norm and yields an infinite family of spatially smooth solutions sharing the same initial data. Unlike the mechanism in \cite{CP}, the mechanism  in  \cite{CDP} is designed to work for any smooth initial data and generates non-uniqueness through an inverse energy cascade. The earlier  work on non-uniqueness for the Navier-Stokes equations is relied on  the Jia–\v{S}ver\'{a}k program and the convex integration argument.

Jia and \v{S}ver\'{a}k in \cite{JS14, JS} proposed a distinct and elegant analytical framework for studying the non-uniqueness of solutions to the Navier–Stokes equations. Their central insight is to reformulate the problem in self-similar variables, thereby translating the problem of solution branching at the initial time into one concerning the nonlinear stability of stationary profiles. Albritton, Bru\'{e} and Colombo in \cite{ABC} proved the non-uniqueness of the Leray-Hopf weak solutions for the forced 3D Navier-Stokes equations by adapting properly Vishik's
construction in \cite{V18a, V18b},  in accordance with the predictions of Jia and \v{S}ver\'{a}k. In a recent paper \cite{HWY}, Hou, Wang, and Yang claim to prove the non-uniqueness of Leray–Hopf solutions to the unforced 3D Navier–Stokes equations through a computer-assisted verification of the key spectral properties.

Convex integration method originates in Nash’s $C^1$ isometric embedding theorem \cite{Nash}. It was first introduced into the Euler equations in the seminal works of De Lellis and Sz\'{e}kelyhidi \cite{DS09, DS13} and further developed in a series of works (see, e.g., \cite{Buc, BDIS15, BDS, DS09, DS14, DS17, DRS,  EPP2,  Ise17, Ise18, Ise22}), and was subsequently extended to study the non-uniqueness of solutions for the Navier-Stokes equations. The breakthrough work of Buckmaster and Vicol \cite{BV} established the first example of non-unique weak solutions for the Navier-Stokes equations in the space $C_tL^2(\mathbb{T}^3)$. Subsequently, in joint work with Colombo \cite{BCV}, they established non-uniqueness for a class of weak solutions that  have bounded
kinetic energy, integrable vorticity, and are smooth outside a fractal set of singular times with Hausdorff dimension strictly
less than 1.  Cheskidov and Luo \cite{1Cheskidov} leveraged temporal intermittency within the convex integration framework to establish the nonuniqueness of very weak solutions in $L^p_tL^q(\mathbb{T}^3)$ for $1 \le p < 2$, demonstrating the sharpness of the Ladyzhenskaya–Prodi–Serrin condition at the endpoint $(p,q) = (2,\infty)$. Subsequently, the authors in \cite{CL23} showed  the non-uniqueness  in $C_tL^p(\TTT^2)$ for $p<2$, and thus sharpened the corresponding uniqueness result established in \cite{FLT}. By developing an iterative scheme that refines the approximate solution via a decomposition into local and non-local parts, we recently proved the non-uniqness of solutions in the class of $C_tL^2(\R^3)$ and $L^p_tL^q(\mathbb{R}^3)$ for $1 \le p < 2$ in \cite{MNY-ar1,MNY-ar2}.  Convex integration argument has also been applied to various other models, such as  the MHD equations \cite{2Beekie, EPP, LZZ, MY, MNY2} and the Boussinesq equations \cite{K26, MNY, TZ17, TZ18}.

\noindent\textbf{Organization of the paper.} In Section~\ref{Induction}, we introduce a new iterative scheme and an iterative proposition, which are crucial for proving  Theorem~\ref{t:main}. In Section \ref{proof-1}, we give the proof of  the iterative proposition.  In Section \ref{p-t}, we complete the proof of Theorem~\ref{t:main} based on the iterative proposition. Finally, the Appendix collects the necessary mathematical tools required in this paper.
 
\noindent\textbf{Notation.} In this paper, $\TTT^2=\R^2/(2\pi\ZZ)^2$. For a $\TTT^2$-periodic function $f$, we denote
\begin{align*}
\mathbb{P}_{=0} f:=\frac{1}{|\TTT^2|}\int_{\TTT^2} f(x)\dd x,\quad\PP f:=f-\mathbb{P}_{=0} f.
\end{align*}
In the following, the notation $a\lesssim b$ means $a\le Cb$ for a universal constant $C$ that may change from line to line.  {Without ambiguity, we will denote $L^r([0,T];Y(\TTT^2))$ and $L^r([0,\infty);Y(\TTT^2))$ by $L^r_T Y$ and $L^r_t Y$,  respectively.} For a vector funtion $f=(f_1, f_2)$, $\curl f=\partial_1 f_2-\partial_2 f_1$.

Let $\partial^{\sigma}$ be the space derivatives, we denote
\begin{align*}
\|f\|_{L^r_tC^N}:=\sum_{j=0}^N\max_{|\sigma|=j}\|\partial^\sigma f\|_{L^r_{t}L^\infty}.
\end{align*}
and 
\begin{align*}
\|(f_1,f_2,\cdots,f_m)\|_{X}:=\max\Big\{\|f_1\|_X, \|f_2\|_X, \cdots, \|f_m\|_X\Big\}.
\end{align*}
\section{Induction scheme}\label{Induction}
The construction of the solutions in Theorem~\ref{t:main} is inspired by the non-uniqueness mechanism developed in~\cite{CP}. To obtain two distinct global solutions in ${\rm BMO}^{-1}$ sharing with the same initial datum, the two sequences of  solutions are constructed  via an alternating and intertwined scheme. To present this construction in a clear and structured manner, we introduce an iterative scheme which is  formally stated as Proposition~\ref{iteration}.
\subsection{Parameters}\label{para}
First of all, we introduce several parameters. Since \(\Lambda\subseteq \mathbb{S}^1\cap\mathbb{Q}^2\) defined in Lemma~\ref{first S} is a finite set, there exists a positive integer \(N_{\Lambda}\) such that 
\[N_{\Lambda} k \in \mathbb{Z}^2, \,\,\quad\forall k\in \Lambda.\]
 Let \(C_0\) be a sufficiently large constant. Suppose that \(a,b\in\mathbb{N}\) are two large numbers depending on \(C_0\), with \(a,b\ge 2^{15}\). We define
\begin{align}\label{def-lq}
    \lambda_q := N^4_{\Lambda} a^{b^q}, \quad \ell_q:=\lambda^{-50}_q,\quad q\ge 1.
\end{align}

%\subsection{The Navier-Stokes-Reynolds system}Let us introduce the  relaxation of the system \eqref{NS} with a stress tensor error term that  tends to $0$ in the sense of distributions. More precisely, the approximate system, the so-called Navier–Stokes-Reynolds  system is governed by
%\begin{equation}\label{NSR}\tag{NSR}
%\left\{ \begin{alignedat}{-1}
%   & \del_t u_q-\Delta u_q+\Div (u_q\otimes u_q)  +\nabla p_q   =\RR_q,
% \\
 % &\nabla \cdot u_q = 0, \\
%\end{alignedat}\right.
%\end{equation}
%associated with some initial data $\uin\in H^3(\R^3)$.
%The \emph{Reynolds stress} $\RR_q$ is trace-free symmetric matrix. Here and below,  $v \otimes u\coloneq  (v_i u_j)_{i,j=1}^3$, and the divergence of a $3\times 3$ matrix $M=(M_{ij})_{i,j=1}^3$ is defined by $\Div M$ with components $(\Div M)_j \coloneq   \partial_i M_{ij}$.
%We enforce the conditions on $u_q$ and $p_q$:
%\[\int_{\R^3}u_q\dd x=\int_{\R^3}p_q\dd x=0.\]
\subsection{Iterative procedure }\label{sec-ite} We present our iteration by induction. Define \((u_j,p_j)=(0,0)\) for \(j\le 0\), and set
\[
(u_1,p_1)=\bigl(\lambda_1e^{-\lambda_1^2 t}\sin(\lambda_1 x_1)\,e_2,\;0\bigr),
\qquad 
\wpp_1 = u_1,\qquad \wss_1 = w_1^{\mathrm{(ns)}} = 0.
\]

Given an integer $m\ge 2$, assume that for each $0\le j\le m-1$, the  solutions  
$\{(u_j,p_j)\}_{0\le j\le m-1}\subseteq C^\infty_{t,x}(\R^{+}\times \TTT^2)$ of the Navier–Stokes equations \eqref{NS} on \(\mathbb{R}^{+}\times\mathbb{T}^2\)  have been constructed and satisfy 
\begin{align}
& \|u_j\|_{L^{\infty}_{t,x}} \le C_0\lambda_j,
   \qquad 
   \|u_j\|_{L^1_t C^2} \le C_0\lambda_j,  \label{u-Linfty} \\[4pt]
& \|u_j\|_{L^1_t C^1} \le 
    \begin{cases}
        C_0 q,      & \text{if } j = 2q, \\[2pt]
        C_0 (q+1), & \text{if } j = 2q+1,
    \end{cases} \label{u-L1C1} \\[8pt]
& u_j = 
    \begin{cases}
        \displaystyle\sum_{l=1}^{q} w_{2l}
        = \sum_{l=1}^{q} \bigl( \wpp_{2l} + \wss_{2l} + w_{2l}^{\mathrm{(ns)}} \bigr), 
        & \text{if } j = 2q, \\[10pt]
        \displaystyle\sum_{l=0}^{q} w_{2l+1}
        = \sum_{l=0}^{q} \bigl( \wpp_{2l+1} + \wss_{2l+1} + w_{2l+1}^{\mathrm{(ns)}} \bigr),
        & \text{if } j = 2q+1.
    \end{cases} \label{u-decomposition}
\end{align}
Let $\mathcal R := (\nabla+\nabla^{\!\top})(-\Delta)^{-1}$, we rewrite $\wpp_j$ and $\wss_j$ as
\begin{align}\label{def-R}
 \wpp_j = \Div(\mathcal R \wpp_j),\,\,\wss_j
 = \Div(\mathcal R \wss_j)   
\end{align}
such that 
\begin{align}
&\|\partial_t^{\,l}\, \mathcal R \wpp_j\|_{L^\infty_t C^{N}} 
    \le C^{3/4}_0 \lambda_j^{N+2l}, \,\quad\|\mathcal R\wss_j\|_{L^\infty_t C^{N}} \le C^{3/4}_0 \lambda_{j-1}^{N},\quad\quad\,\,\,0\le N\le 4;\,\,l=0,1,\label{Rwj}\\
    &\|\wpp_j\|_{L^\infty_t C^{N}} \le C^{3/4}_0 \lambda_j^{N+1}, \qquad\quad\,\,\|\wss_j\|_{L^\infty_t C^{N}} \le C^{3/4}_0 \lambda_{j-1}^{N+1},\qquad\quad
    0\le N\le 3, \label{wp-N}\\
    &\|w^{(\rm p)}_{j}\|_{L^1_tC^N} \leq C^{3/4}_0 \lambda^{N-1}_{j}, \qquad \,\,\,\quad\|w^{(\rm s)}_{j}\|_{L^1_tC^N} \leq C^{3/4}_0 \lambda^{-2}_{j} \lambda^{N+1}_{j-1}, \,\quad\,\,\,0\le N\le 2, \label{wp-wsL1}\\
    & \|\wpp_j\|_{L^{2}_tL^r \cap L^{1}_t W^{1,r}} \le C_0^{3/4}\, 2^{-\frac{j}{r}}, \quad\,\,\,\|\wss_j\|_{L^{2}_tL^r \cap L^{1}_t W^{1,r}} \le C_0^{3/4} 2^{-\frac{j}{r}},\,\,\,\,\,\,
    \forall\, 1\le r<\infty. \label{wp-L}
\end{align}
The perturbation flow $w^{\text{(ns)}}_j$ satisfy that
\begin{align}\label{e-ns}
  \|w^{\text{(ns)}}_j\|_{\widetilde L^{\infty}_tB^{-1/2}_{\infty,1}\cap L^{1}_tB^{3/2}_{\infty,1}} \le \lambda^{-10}_{j-1},\qquad 
\|w^{\text{(ns)}}_j\|_{\widetilde L^{\infty}_tB^{0}_{\infty,1}\cap L^{1}_tB^{2}_{\infty,1}} \le \lambda^{1/2}_{j}.
\end{align}

\begin{proposition}[Inductive step]\label{iteration} There exists an absolute constant $C_0$ (independent of $m$) with the following property. Let $m \ge 2$ and suppose that $\{(u_j, p_j)\}_{0 \le j \le m-1} \subseteq C^\infty_{t,x}(\mathbb{R}^+ \times \mathbb{T}^2)$ are solutions of the Navier–Stokes system \eqref{NS} satisfying estimates \eqref{u-Linfty}–\eqref{wp-L}. Then there exists a  $w_m = w^{(\rm p)}_{m} + w^{(\rm s)}_{m} + w^{(\rm ns)}_{2q}$ such that  $u_m := u_{m-2} + w_m$ also solves \eqref{NS} and satisfies  \eqref{u-Linfty}–\eqref{wp-L} with $m-1$ replaced by $m$, and
\begin{align}
& \|u_m\|_{L^\infty_t\BMO^{-1}} \le C_0, \label{e-Um-BMO}\\
& w^{\rm{(p)}}_{m-1}(0,x) = w^{\rm{(s)}}_m(0,x), \qquad w^{\rm{(ns)}}_m(0,x) = 0. \label{ini-wm}
\end{align}
Moreover, the above estimates immediately show that
\begin{align}
    & \|w_m\|_{L^2_t L^r \cap L^1_t W^{1,r}} \le C_0 2^{-\frac{m}{r}}, \qquad \forall\, 1 \le r < \infty, \label{wm-L} \\
& \|w_m\|_{L^\infty_{t,x}} \le \tfrac12 C_0 \lambda_m, \quad
   \|w_m\|_{L^1_t C^1} \le \tfrac12 C_0,\quad
   \|w_m\|_{L^1_t C^2} \le \tfrac12 C_0 \lambda_m.\label{wm-infty} 
\end{align}
\end{proposition}

%Moreover, we will show the following proposition.
%\begin{proposition}\label{p:main-prop2}Let $T>0$ and $e(t)=\widetilde{e}(t)$ for $t\in[0, \tfrac{T}{4}+\lambda^{-1}_1  ]$. Suppose that  $(u_q ,p_q,\RR_q)$ solves
%\eqref{NSR} and satisfies  \eqref{uq-tigh}--\eqref{et},    $(\widetilde{u}_q, \widetilde{p}_q, \widetilde{\RR}_q)$  solves
%\eqref{NSR} and satisfies  \eqref{uq-tigh}--\eqref{et} with $e(t)$ replaced by $\widetilde{e}(t)$. Then   if
%\begin{align*}
 %\uql=\tuql, \quad\uqnl=\tuqnl, \quad \RR_{q } ={\tRR}_{q },\quad\text{on}\quad [0, \tfrac{T}{4}+\lambda^{-1}_q],
%\end{align*}
%then we have
%\begin{align*}
 %\uqql=\tuqql, \quad\uqqnl=\tuqqnl, \quad \RR_{q+1 } ={\tRR}_{q+1},\quad\text{on}\quad [0, \tfrac{T}{4}+\lambda^{-1}_{q+1}].
%\end{align*}
%\end{proposition}

\section{Proof of Proposition \ref{iteration}}\label{proof-1}
In this section, we prove Proposition~\ref{iteration}. Without loss of generality, we consider the case where $m$ is even, i.e., $m=2q$ for some $q\in\NN$. Our  purpose is to construct  $w_{2q}$, which consists of three components: the heat-dominated Fourier mode flow $\wpp_{2q} $, the inverse cascade-dominated flow $\wss_{2q}$ and the perturbation flow $\wtq$.

\subsection{Construction of  $\wpp_{2q}$ and $\wss_{2q}$.}\label{Con-wpws}
Let $\mathbb{T} = \mathbb{R} / 2\pi\mathbb{Z}$ denote the one-dimensional torus. We consider a smooth, even function $\phi: \mathbb{T} \to \mathbb{R}$ with $\int_{\TTT}\phi^2\dd x=1$ and whose spatial support satisfies
\[
\operatorname{spt}_x \phi \subseteq \bigcup_{m \in \mathbb{Z}} \bigl([-\tfrac{1}{100}, \tfrac{1}{100}] + 2\pi m\bigr).
\]
For each frequency parameter $\lambda_l$ and vectors $k\in \Lambda$ defined in Lemma \ref{first S}, we define the building blocks
\[
\phi_{k,l}(x) = \phi\bigl(  \lambda^{1/4}_{l} \, k \cdot x \bigr), \qquad l \in \mathbb{N}^{+}, \, k \in \Lambda. 
\]
 Moreover, $\mathbb{P}_{= 0}(\phi^2_{k,l})=1.$ 
We denote 
\begin{align}\label{def-Omega}
\Omega_{k,l} := \operatorname{spt} \, \phi_{k,l}, \qquad \widetilde{\Omega}_{k,l} := \bigl\{ x : d(x, \Omega_{k,l}) < \tfrac{1}{100} \lambda_l^{-1/4} \bigr\},\,\text{for}\,\,l \in \mathbb{N}^+,\,k \in \Lambda,    
\end{align}
 and 
\[
\Omega_{2q-1} = \bigcap_{1 \le l \le 2q-1} \, \bigcup_{k \in \Lambda} \Omega_{k,l}, \qquad
\widetilde{\Omega}_{2q-1} = \bigcap_{1 \le l \le 2q-1} \, \bigcup_{k \in \Lambda} \widetilde{\Omega}_{k,l}.
\]
Finally, we introduce a smooth cutoff function $\chi_{2q-1}(x)$ satisfying
\begin{align}\label{def-chi}
\chi_{2q-1}(x)=1 \,\text{if}\,x\in \Omega_{2q-1}; \quad \chi_{2q-1}(x)=0 \,\text{if}\,\, x\notin\widetilde{\Omega}_{2q-1}; \quad \|\chi_{2q-1}\|_{C^N} \lesssim \lambda^{N/4}_{2q-1}.
\end{align}

\noindent \textbf{Amplitudes.}
Using the coefficients $a_k$
  in Lemma~\ref{first S}, we now define the amplitudes $a_{(k,2q)}$. Let  $\psi_{\ell_{2q-1}}$ and $\varphi_{\ell_{2q-1}}$ be the mollifiers defined in Definition \ref{e:defn-mollifier-t}. The amplitude associated with the direction $k \in \Lambda$ is given by
\begin{align}\label{def-akq}
a_{(k,2q)}(t,x) =& (2000C_0)^{1/2} \int_{t}^{t+\ell_{2q-1}} \Bigl( a_k \bigl( \mathrm{Id} + \tfrac{ \mathcal{R}w^{(p)}_{2q-1} }{ 1000C_0} \bigr) \ast \psi_{\ell_{2q-1}} \Bigr)(x,s) \, \varphi_{\ell_{2q-1}}(t-s) \, \dd s\notag\\
=:& (2000C_0)^{1/2}\big(a_k \bigl( \mathrm{Id} + \tfrac{ \mathcal{R}w^{(p)}_{2q-1} }{ 1000C_0} \bigr) \big)\ast_{t,x}{\psi}_{\ell_{2q-1}}\varphi_{\ell_{2q-1}},
\end{align}
where $C_0 > 0$ is a universal constant, and $\ell_{2q-1}$ is defined by \eqref{def-lq}. From the construction of the amplitude and the building blocks, it is easy to verify the following estimates.

\begin{proposition}\label{def-ak}
For any integer $N \ge 0$, the amplitude functions and building blocks satisfy
    \begin{align*}
        &\|a_{(k,2q)}\|_{L^\infty_t C^N} \lesssim C_0^{1/2} \, \ell^{-N}_{2q-1},\,\,\|\phi_{k,2q}\|_{C^N} \lesssim \lambda^{N/4}_{2q}.
    \end{align*}
\end{proposition}
Let $\nabla^{\perp}=(\partial_{2}, -\partial_{1})$.  We construct the heat-dominated Fourier mode flow:
\begin{equation}\label{def-wpq-1}
\begin{aligned}
\wpp_{2q} = &-\lambda^{-3}_{2q} e^{-\lambda^2_{2q }t} \, \Delta \, \nabla^{\perp} \Bigl( \sum_{k \in \Lambda} a_{(k,2q)} \chi_{2q-1} \phi_{k,2q} \, \operatorname{curl} \bigl( e^{\mathrm{i}\lambda_{2q } k \cdot x} \bar{k} \bigr) \Bigr) \\
=&-\lambda^{-3}_{2q }e^{-\lambda^2_{2q }t}\Div(\nabla+\nabla^{\TT}){\nabla^{\perp}}\Big(\sum_{k\in\Lambda}a_{(k,2q)} \chi_{2q-1}\phi_{k,2q}\curl \big( e^{\ii\lambda_{2q }k\cdot x}\bar{k}\big)\Big)\\
 =&\Div (\mathcal{R}\wpp_{2q}).
\end{aligned}
\end{equation}
 Moreover, for $k \perp \bar{k}$ and $|k|=1$, we have
 $$-\Delta( e^{\mathrm{i}\lambda_{2q } k \cdot x} \bar k ) =\nabla^{\perp}\curl ( e^{\mathrm{i}\lambda_{2q } k \cdot x} \bar k ) =\lambda^2_{2q} e^{\mathrm{i}\lambda_{2q } k \cdot x} \bar k,$$
 so that we  decompose $\wpp_{2q}$ into the main and remainder parts:
\begin{align}
\wpp_{2q}
=& \underbrace{ \lambda_{2q} e^{-\lambda^2_{2q}t} \sum_{k \in \Lambda} a_{(k,2q)} \chi_{2q-1} \phi_{k,2q} e^{\mathrm{i}\lambda_{2q } k \cdot x} \bar{k} }_{\wpm_{2q}} \notag \\
& - \lambda^{-1}_{2q } e^{-\lambda^2_{2q}t} \sum_{k \in \Lambda} \Delta\bigl( a_{(k,2q)} \chi_{2q-1} \phi_{k,2q} \bigr) e^{\mathrm{i}\lambda_{2q } k \cdot x} \bar{k} \notag \\
& - 2\mathrm{i} e^{-\lambda^2_{2q }t} \sum_{k \in \Lambda} \nabla\bigl( a_{(k,2q)} \chi_{2q-1} \phi_{k,2q} \bigr) \cdot k \, e^{\mathrm{i}\lambda_{2q } k \cdot x} \bar{k} \notag \\
& - \lambda^{-3}_{2q } e^{-\lambda^2_{2q }t} \Delta \Bigl( \sum_{k \in \Lambda} \nabla^{\perp}\bigl( a_{(k,2q)} \chi_{2q-1} \phi_{k,2q} \bigr) \, \operatorname{curl}\bigl( e^{\mathrm{i}\lambda_{2q } k \cdot x} \bar{k} \bigr) \Bigr) \notag \\
=:& \,\, \wpm_{2q} + \wpr_{2q}. \label{dec-wp}
\end{align}

Next, we construct $\wss_{2q}$ to cancel the low frequency part  of the quadratic interaction $\Div(\wpm_{2q} \otimes \wpm_{2q})$. To be more precise, we set
\begin{equation}\label{wttq}
\begin{aligned}
\wss_{2q} = \frac{1}{2} \chi^2_{2q-1} \sum_{k \in \Lambda} \Div \Bigl( a^2_k \bigl( \mathrm{Id} + \tfrac{ \mathcal{R}\wpp_{2q-1} }{ 1000C_0} \bigr) (2000C_0) \, \bar{k} \otimes \bar{k} \Bigr) e^{-2\lambda^2_{2q }t}.
\end{aligned}
\end{equation}
Utilizing Lemma \ref{first S}  and  $\chi_{2q-1}\wpp_{2q-1}=\wpp_{2q-1}$, we obtain 
\begin{align}
\wss_{2q} =& \frac{1}{2} \chi^2_{2q-1} \Div\bigl( \mathrm{Id} + \tfrac{\mathcal{R}\wpp_{2q-1}}{1000C_0} \bigr) (2000C_0) e^{-2\lambda^2_{2q }t} \notag \\
=& \chi^2_{2q-1} \, \wpp_{2q-1} \, e^{-2\lambda^2_{2q }t} = \wpp_{2q-1} \, {e^{-2\lambda^2_{2q }t}}. \label{def-ws}
\end{align}
This equality shows that
\begin{align}\label{wss2}
\wpp_{2q-1}(0,x) = \wss_{2q}(0,x) \quad \text{and} \quad \partial_t \wss_{2q} = -2\lambda^2_{2q } e^{-2\lambda^2_{2q }t} \wpp_{2q-1} + \partial_t \wpp_{2q-1} \, e^{-2\lambda^2_{2q }t}.
\end{align}

\subsection{Key estimates for $\wpp_{2q}$ and $\wss_{2q}$.}\label{est-wpws}

\begin{proposition}[Estimates for $\wpp_{2q}$ and $\wss_{2q}$]\label{Est-wpws}
There exists a universal constant $C_0$ such that 
\begin{align}
&\|\partial^l_t \mathcal{R}w^{(\rm p)}_{2q}\|_{L^{\infty}_{t}C^N} \leq C^{3/4}_0 \lambda^{N+2l}_{2q}, \quad \|\mathcal{R}w^{(\rm s)}_{2q}\|_{L^{\infty}_{t}C^N} \leq C^{3/4}_0 \lambda^N_{2q-1}, \quad \,\, \,\, 0\le N\le 4;l=0,1; \label{wpq-wss1}\\
&\|w^{(\rm p)}_{2q}\|_{L^1_tC^N} \leq C^{3/4}_0 \lambda^{N-1}_{2q}, \qquad \,\,\quad\|w^{(\rm s)}_{2q}\|_{L^1_tC^N} \leq C^{3/4}_0 \lambda^{-2}_{2q} \lambda^{N+1}_{2q-1}, \quad 0\le N\le 2, \label{wpq-wss2}\\
&\|w^{(\rm p)}_{2q}\|_{L^{\infty}_{t}C^N} \leq C^{3/4}_0 \lambda^{N+1}_{2q}, \qquad \,\quad\|w^{(\rm s)}_{2q}\|_{L^{\infty}_{t}C^N} \leq C^{3/4}_0 \lambda^{N+1}_{2q-1}, \qquad\, 0\le N\le 3. \label{wpq-wss3}
\end{align}
\end{proposition}

\begin{proof}
 Starting from the definition \eqref{def-wpq-1}, we have
\begin{align}
\|\mathcal{R}\wpp_{2q}\|_{L^\infty_{t}C^N}
\le& \lambda^{-3}_{2q } \Bigl\| e^{-\lambda^2_{2q }t} (\nabla+\nabla^{\mathsf{T}}) \Bigl( \sum_{k \in \Lambda} {\nabla^{\perp}}(a_{(k,2q)} \chi_{2q-1} \phi_{k,2q}) \operatorname{curl} \bigl( \mathrm{i} e^{\mathrm{i}\lambda_{2q } k \cdot x} \bar{k} \bigr) \Bigr) \Bigr\|_{L^\infty_{t}C^N} \nonumber \\
&+ \lambda^{-1}_{2q } \Bigl\| e^{-\lambda^2_{2q }t} (\nabla+\nabla^{\mathsf{T}}) \Bigl( \sum_{k \in \Lambda} a_{(k,2q)} \chi_{2q-1} \phi_{k,2q} e^{\mathrm{i}\lambda_{2q } k \cdot x} \bar{k} \Bigr) \Bigr\|_{L^\infty_{t}C^N}. \label{est-Rw}
\end{align}
Thanks to \eqref{def-chi} and Proposition \ref{def-ak}, we have
 \begin{align*}
   &\lambda^{-3}_{2q}\Big\|e^{-\lambda^2_{2q }t}  (\nabla+\nabla^{\TT})  \Big(\sum_{k\in\Lambda}{\nabla^{\perp}}(\partial^l_ta_{(k,2q)} \chi_{2q-1}\phi_{k,2q})\curl \big(\ii e^{\ii\lambda_{2q }k\cdot x}\bar{k}\big)\Big)\Big\|_{L^\infty_{t}C^N}\\
   \lesssim&\sum_{k\in\Lambda}\lambda^{-2}_{2q}\|a_{(k,2q)} \chi_{2q-1}\phi_{k,2q}\|_{L^\infty_tC^{N+3}}+C^{1/2}_0\lambda^{N-1}_{2q}\lesssim C^{1/2}_0\lambda^{N-1}_{2q}
 \end{align*}
 and 
  \begin{align*}
&\lambda^{-1}_{2q}\Big\|e^{-\lambda^2_{2q }t}  (\nabla+\nabla^{\TT})  \Big(\sum_{k\in\Lambda}a_{(k,2q)} \chi_{2q-1}\phi_{k,2q} e^{\ii\lambda_{2q }k\cdot x}\bar{k}\Big)\Big\|_{L^\infty_{t}C^N}\\
\lesssim&\sum_{k\in\Lambda}C_0^{1/2}\lambda^{-1}_{2q}\|a_{(k,2q)}\chi_{2q-1}\phi_{k,2q}\|_{L^\infty C^{N+1}}+C^{1/2}_0\lambda^N_{2q}\\
\lesssim&C_0^{1/2}\lambda^{\frac{N}{4}-\frac{3}{4}}_{2q}+C^{1/2}_0\lambda^N_{2q}\lesssim C^{1/2}_0\lambda^N_{2q}.
   \end{align*}
Inserting these estimates into \eqref{est-Rw} yields $\|\mathcal{R}\wpp_{2q}\|_{L^\infty_{t}C^N} \lesssim C^{1/2}_0 \lambda^{N}_{2q}$. In the same way, we obtain the estimate for  $\|\partial_t\mathcal{R}\wpp_{2q}\|_{L^\infty_{t}C^N}$.

By the definition of $\wpp_{2q}$ in \eqref{def-wpq-1}, we have, for $0\le N\le 2$,
\begin{align*}
    \|\wpp_{2q}\|_{L^1_tC^N}\le& \lambda^{-3}_{2q}\|e^{-\lambda^2_{2q}t}\|_{L^1_t}\Big(\|a_{(k,2q)}\|_{L^\infty_tC^{N+3}}+\|\chi_{2q-1}\|_{C^{N+3}}+\|\phi_{k,2q}\|_{C^{N+3}}+\|e^{\ii\lambda_{2q}k\cdot x}\|_{C^{N+4}}\Big)\\
    \lesssim&\lambda^{-5}_{2q}(C_0^{1/2}\ell^{-N-3}_{2q-1}+\lambda^{N+4}_{2q})\le C^{3/4}_0\lambda^{N-1}_{2q},
\end{align*}
and for $0\le N\le 3$,
\begin{align*}
\|\wpp_{2q}\|_{L^\infty_tC^N}\lesssim&\lambda^{-3}_{2q}\sum_{k\in\Lambda}\|a_{(k,2q)}\chi_{2q-1}\phi_{k,2q}\curl(e^{\ii\lambda_{2q}k\cdot x}\bar{k})\|_{L^\infty C^{N+3}}\\
\lesssim&\lambda^{-2}_{2q}\sum_{k\in\Lambda}(\|a_{(k,2q)}\chi_{2q-1}\phi_{k,2q}\|_{L^\infty_tC^{N+3}}+C^{1/2}_0\lambda^{N+3}_{2q})\\
\le&C^{3/4}_0\lambda^{N+1}_{2q}.
\end{align*}
In conclusion, we prove the estimates for $\wpp_{2q}$ in \eqref{wpq-wss1}--\eqref{wpq-wss3}. The bounds for $\wss_{2q}$ in \eqref{wpq-wss1}-- \eqref{wpq-wss3} follow directly from the identity \eqref{def-ws} and the inductive bounds  for $\wpp_{2q-1}$ in \eqref{Rwj} and \eqref{wp-N}.
Therefore, we complete the proof of Proposition \ref{Est-wpws}.
\end{proof}
 \begin{proposition}\label{cedu}For $1\le r<\infty$, we have
 \begin{align*}
 \|w^{\rm (p)}_{2q}\|_{L^{2}_tL^r \cap L^{1}  W^{1,r}}\le C^{\frac{3}{4}}_02^{-\frac{2q}{r}},\,\,\|w^{\rm (s)}_{2q}\|_{L^{2}_tL^r \cap L^{1}  W^{1,r}}\leq C_0^{\tfrac{3}{4}} 2^{-\tfrac{2q}{r}}.    
 \end{align*}
 \end{proposition}
 \begin{proof}
 From the definition of $\wpp_{2q}$ in \eqref{def-wpq-1}, we have
 \begin{align*}
     \spt_x\wpp_{2q}=\spt_x\chi_{2q-1}\subseteq\widetilde{\Omega}_{2q-1}=\bigcap_{\substack{1\le j \le {2q-1}}} \bigcup_{\substack{k \in \Lambda}} \widetilde{\Omega}_{j,k}.
 \end{align*}
where $\widetilde{\Omega}_{j,k}$ is defined in \eqref{def-Omega}. 
According to the definition of \(\phi_{k,j}\), one easily verifies that the support of \(\phi_{k,j}\) consists of \(\lambda_j^{1/4}\) parallelograms on the periodic domain $\TTT^2$, and each of size is less than \(\frac{1}{4}\lambda_j^{-1/4}\). Consequently, we obtain  
\[
|\Omega_{j}| \le \frac{1}{4}\,\lambda_{j}^{-1/4}\; \frac{|\Omega_{j-1}|}{|\TTT^2|}\; \lambda_j^{1/4}=\frac{1}{4}{|\Omega_{j-1}|}\le4^{-j}|\TTT^2|.
\]
Similarly, one deduces that
\begin{align*}
    |\widetilde{\Omega}_{2q-1}|\le 2^{-2q+1}|\TTT^2|.
\end{align*}
Noting that
 $\supp  \chi_{2q-1}\subseteq\widetilde{\Omega}_{2q-1}$, we deduce that
 \begin{align*}
 &\|\wpp_{2q}\|_{L^{2}_tL^r \cap L^{1}_t  W^{1,r}}\\
 \leq&
 \Big\|\lambda^{-3}_{2q }e^{-\lambda^2_{2q }t}\Delta{\nabla^{\perp}}\Big(\sum_{k\in\Lambda}a_{(k,2q)} \chi_{2q-1}\phi_{k,2q}\curl \big( e^{\ii\lambda_{2q }k\cdot x}\bar{k}\big)\Big)\Big\|_{L^2L^r}\\
 &+\Big\|\lambda^{-3}_{2q }e^{-\lambda^2_{2q }t}\Delta{\nabla^{\perp}}\Big(\sum_{k\in\Lambda}a_{(k,2q)} \chi_{2q-1}\phi_{k,2q}\curl \big( e^{\ii\lambda_{2q }k\cdot x}\bar{k}\big)\Big)\Big\|_{ L^{1}_t  W^{1,r}}
 \\
 \lesssim&\lambda^{-4}_{2q }\big(\lambda_{2q}\|a_{(k,q)}\|_{L^{\infty}_{t}C^3} \|\chi_{2q-1}\phi_{k,2q}\|_{L^{r} }+\lambda_{2q}C^{1/2}_0 \|\chi_{2q-1}\phi_{k,2q}\|_{W^{3,r} }+C^{1/2}_0\lambda^4_{2q}\|\chi_{2q-1}\phi_{k,2q}\|_{L^{r} }\big)\\
 &+\lambda^{-5}_{2q }\big(\lambda_{2q}\|a_{(k,q)}\|_{L^{\infty}_{t}C^4} \|\chi_{2q-1}\phi_{k,2q}\|_{L^{r} }+\lambda_{2q}C^{1/2}_0 \|\chi_{2q-1}\phi_{k,2q}\|_{W^{4,r} }+C^{1/2}_0\lambda^5_{2q}\|\chi_{2q-1}\phi_{k,2q}\|_{L^{r} }\big)\\
 \lesssim& C^{1/2}_0|\widetilde{\Omega}_{2q-1}|^{\frac{1}{r}}(\lambda^{-3}_{2q}\ell^{-3}_{2q-1}+\lambda^{-\frac{9}{4}}_{2q}+1)\\
 \lesssim& C^{1/2}_02^{-\frac{2q-1}{r}}.
 \end{align*}
 Similarly, we also have $\|\wss_{2q}\|_{_{L^{2}_tL^r \cap L^{1}_t  W^{1,r}}}\le  C^{\frac{3}{4}}_02^{-\frac{2q }{r}} ,$  and thus complete the proof of Proposition \ref{cedu}.
 \end{proof}

 \begin{proposition}\label{BMO}There exists a universal $C_0 $ such that 
\begin{align*}
   \big\|\sum_{l=1}^q(w^{\rm (p)}_{2l}+w^{\rm (s)}_{2l})\big\|_{L^\infty {\rm BMO}^{-1}}\le \frac{1}{16}{C_0}\sum_{m=-1}^{2q}2^{-m}.
 \end{align*}     
 \end{proposition}
We refer to \cite[Proposition 3.8]{CP} for the proof of this proposition.

\subsection{ Estimates of the error terms}
\begin{proposition}\label{def-F2}
Let $w^{\rm (p)}_{2q}$ and $w^{\rm (s)}_{2q}$ be defined in \eqref{def-wpq-1} and \eqref{def-ws}. Then there exist a pressure $P^{(1)}_{2q}$ and an error term $F^{(1)}_{2q}$ such that
\begin{align}
    \Div(w^{\rm (p)}_{2q} \otimes w^{\rm (p)}_{2q}) + \partial_t w^{\rm (s)}_{2q} = F^{(1)}_{2q} + \nabla P^{(1)}_{2q},
    \label{decom-wpwp}
\end{align}
with the following estimates:
\begin{align}
    \|F^{(1)}_{2q}\|_{L^1_t B^{-1/2}_{\infty,1}} \lesssim C_0 \lambda^{-15}_{2q-1}, \qquad
    \|F^{(1)}_{2q}\|_{L^1_t B^{0}_{\infty,1}}    \lesssim C_0 \lambda^{\frac{1}{4}}_{2q}. \label{est-Fq1-2}
\end{align}
\end{proposition}

\begin{proof}
Using the decomposition \eqref{dec-wp}, we write
\begin{align*}
    &\Div(\wpp_{2q} \otimes \wpp_{2q})\\
    = &\Div(\wpm_{2q} \otimes \wpm_{2q}) 
       + \Div\bigl( \wpm_{2q} \otimes \wpr_{2q} 
                   + \wpr_{2q} \otimes \wpm_{2q} 
                   + \wpr_{2q} \otimes \wpr_{2q} \bigr) \\
    =& -\Div\Bigl( \lambda^2_{2q} e^{-2\lambda^2_{2q}t} 
                   \sum_{k,k' \in \Lambda} 
                   \chi^2_{2q-1} a_{(k,2q)} a_{(k',2q)} 
                   \phi_{k,2q} \phi_{k',2q} 
                   e^{\ii\lambda_{2q}(k+k')\cdot x} 
                   \bar{k} \otimes \bar{k'} \Bigr) \\
    &+ \Div\bigl( \wpm_{2q} \otimes \wpr_{2q} 
                       + \wpr_{2q} \otimes \wpm_{2q} 
                       + \wpr_{2q} \otimes \wpr_{2q} \bigr).
\end{align*}
From Lemmas \ref{Betrimi} and \ref{first S}, a direct computation gives
\begin{align}
&-\Div\Bigl( \lambda^2_{2q} e^{-2\lambda^2_{2q}t} \chi^2_{2q-1} 
             \sum_{k,k'\in\Lambda} a_{(k,2q)} a_{(k',2q)} 
             \phi_{k,2q} \phi_{k',2q} 
             e^{\mathrm{i}\lambda_{2q}(k+k')\cdot x} 
             \bar{k} \otimes \bar{k'} \Bigr)\notag \\
&= \frac{1}{2} \lambda^2_{2q} e^{-2\lambda^2_{2q}t} \mathbb{P}_{\neq 0} \nabla_{\xi}
   \Bigl( \Bigl| \sum_{k \in \Lambda} \chi_{2q-1} a_{(k,2q)} \phi_{k,2q} 
          e^{\mathrm{i}\lambda_{2q} k \cdot \xi} \bar{k} \Bigr|^2
        - \Bigl| \sum_{k \in \Lambda} \chi_{2q-1} a_{(k,2q)} 
          \phi_{k,2q} e^{\mathrm{i}\lambda_{2q} k \cdot \xi} \Bigr|^2 \Bigr)
   \Big|_{\xi=x} \notag \\
&\quad + \lambda^2_{2q} e^{-2\lambda^2_{2q}t} \mathbb{P}_{\neq 0} \Div_{x}
   \Bigl( \sum_{\substack{k'+k\neq0 \\ k',k \in \Lambda}} 
          - a_{(k,2q)} a_{(k',2q)} \phi_{k,2q} \phi_{k',2q} 
            \chi^2_{2q-1} e^{\mathrm{i}\lambda_{2q}(k'+k)\cdot \xi} 
            \bar{k'} \otimes \bar{k} \notag \\
&\qquad\qquad\qquad\qquad\qquad 
          + \sum_{k \in \Lambda} \chi^2_{2q-1} a^2_{(k,2q)} 
            \bar{k} \otimes \bar{k} \Bigr)
   \Big|_{\xi=x} \notag \\
&= \nabla_{x} \Bigl( \sum_{\substack{k'+k\neq0\notag  \\ k,k'\in\Lambda}} 
        \frac{1}{2} \lambda^2_{2q} e^{-2\lambda^2_{2q}t} 
        \chi^2_{2q-1} a_{(k,2q)} a_{(k',2q)} 
        \phi_{k,2q} \phi_{k',2q} 
        \bigl( \bar{k} \cdot \bar{k'} - 1 \bigr) 
        e^{\mathrm{i}\lambda_{2q}(k'+k)\cdot x} \Bigr) \notag \\
&\quad - \mathbb{P}_{\neq 0} \Bigl( \sum_{\substack{k'+k\neq0 \\ k,k'\in\Lambda}} 
        \Bigl[ \frac{1}{2} \nabla\bigl( a_{(k,2q)} a_{(k',2q)} 
               \chi^2_{2q-1} \phi_{k,2q} \phi_{k',2q} \bigr) 
               \bigl( \bar{k} \cdot \bar{k'} - 1 \bigr) \notag \\
&\qquad\qquad\qquad\quad 
        + \Div \bigl( a_{(k,2q)} a_{(k',2q)} 
          \phi_{k,2q} \phi_{k',2q} \bar{k'} \otimes \bar{k} \bigr) \Bigr] 
        \lambda^2_{2q} e^{-2\lambda^2_{2q}t} 
        e^{\mathrm{i}\lambda_{2q}(k'+k)\cdot x} \Bigr) \notag \\
&\quad + \lambda^2_{2q} e^{-2\lambda^2_{2q}t} 
         \Div\Bigl( \sum_{k \in \Lambda} 
                \chi^2_{2q-1} \phi^2_{k,2q} a^2_{(k,2q)} 
                \bar{k} \otimes \bar{k} \Bigr) \notag \\
&=: \nabla_{x} P^{(1,1)}_{2q} + F^{(1,1)}_{2q}
   + \lambda^2_{2q} e^{-2\lambda^2_{2q}t} 
     \Div\Bigl( \sum_{k \in \Lambda} 
            \chi^2_{2q-1} a^2_{(k,2q)} \phi^2_{k,2q} 
            \bar{k} \otimes \bar{k} \Bigr).\label{struc}
\end{align}
By the definition of $a_{(k,2q)}$ in \eqref{def-akq} and Lemma \ref{first S}, we further have
\begin{align*}
&\lambda^2_{2q} e^{-2\lambda^2_{2q}t} 
   \Div\Bigl( \sum_{k \in \Lambda} 
          \chi^2_{2q-1} a^2_{(k,2q)} \phi^2_{k,2q} 
          \bar{k} \otimes \bar{k} \Bigr) \\
=& \lambda^2_{2q} e^{-2\lambda^2_{2q}t} 
   \sum_{k \in \Lambda} \Div\Bigl( 
        2000C_0 \chi^2_{2q-1} \phi^2_{k,2q} 
        \Bigl[ \bigl(  a_k 
               \bigl( \mathrm{Id} + \tfrac{\mathcal{R}\wpp_{2q-1}}{1000C_0} \bigr)\ast_{t,x}\psi_{\ell_{2q-1}}\varphi_{2q-1} \bigr)^2 
             - a^2_k \bigl( \mathrm{Id} 
               + \tfrac{\mathcal{R}\wpp_{2q-1}}{1000C_0} \bigr) \Bigr] 
        \bar{k} \otimes \bar{k} \Bigr) \\
&\quad + \lambda^2_{2q} e^{-2\lambda^2_{2q}t} 
   \sum_{k \in \Lambda} \Div\Bigl( 
    2000C_0 \chi^2_{2q-1} \phi^2_{k,2q} 
        a^2_k \bigl( \mathrm{Id} + \tfrac{\mathcal{R}\wpp_{2q-1}}{1000C_0} \bigr) 
        \bar{k} \otimes \bar{k} \Bigr) \\
=& \lambda^2_{2q} e^{-2\lambda^2_{2q}t} 
   \sum_{k \in \Lambda} \Div\Bigl( 
    2000C_0 \chi^2_{2q-1} \phi^2_{k,2q} 
        \Bigl[ \bigl(   a_k 
               \bigl( \mathrm{Id} + \tfrac{\mathcal{R}\wpp_{2q-1}}{1000C_0} \bigr) \ast_{t,x}\psi_{\ell_{2q-1}}\varphi_{2q-1} \bigr)^2 
             - a^2_k \bigl( \mathrm{Id} 
               + \tfrac{\mathcal{R}\wpp_{2q-1}}{1000C_0} \bigr) \Bigr] 
        \bar{k} \otimes \bar{k} \Bigr) \\
&\quad + \lambda^2_{2q} e^{-2\lambda^2_{2q}t} 
   \sum_{k \in \Lambda} \Div\Bigl( 
        2000C_0 \chi^2_{2q-1} 
        \bigl( \mathbb{P}_{\neq 0} \phi^2_{k,2q} \bigr) 
        a^2_k \bigl( \mathrm{Id} + \tfrac{\mathcal{R}\wpp_{2q-1}}{1000C_0} \bigr) 
        \bar{k} \otimes \bar{k} \Bigr) \\
&\quad + 2\lambda^2_{2q} e^{-2\lambda^2_{2q}t} \wpp_{2q-1} 
   + \nabla \bigl( 2000C_0 \lambda^2_{2q} e^{-2\lambda^2_{2q}t} \chi^2_{2q-1} \bigr) \\
&=: F^{(1,2)}_{2q} + 2\lambda^2_{2q} e^{-2\lambda^2_{2q}t} \wpp_{2q-1} 
   + \nabla P^{(1,2)}_{2q}.
\end{align*}
Now we set
\begin{align*}
    F^{(1)}_{2q} &:= F^{(1,1)}_{2q} + F^{(1,2)}_{2q} 
                   + \partial_t \wpp_{2q-1} e^{-2\lambda^2_{2q}t} \\
    &\quad + \Div\bigl( \wpm_{2q} \otimes \wpr_{2q} 
                       + \wpr_{2q} \otimes \wpm_{2q} 
                       + \wpr_{2q} \otimes \wpr_{2q} \bigr), \\
    P^{(1)}_{2q} &:= P^{(1,1)}_{2q} + P^{(1,2)}_{2q}.
\end{align*}
Using \eqref{wss2}, we obtain
\begin{align*}
    \Div(\wpp_{2q} \otimes \wpp_{2q}) + \partial_t \wss_{2q}
    = \nabla_{x} P^{(1)}_{2q} + F^{(1)}_{2q}.
\end{align*}
It remains to establish the estimate \eqref{est-Fq1-2}. 

\noindent\textbf{Estimates for $F^{(1,1)}_{2q}$.}
By Propositions \ref{def-ak} and \ref{Est-fe}, we have
\begin{align}
  \|F^{(1,1)}_{2q}\|_{L^1_t B^{-1}_{\infty,\infty}}
    &\lesssim \Bigl( \sum_{\substack{k'+k\neq0 \notag\\ k,k'\in\Lambda}} 
               \lambda^{-1}_{2q} 
               \|\chi^2_{2q-1} a_{(k,2q)} a_{(k',2q)} 
                 \phi_{k,2q} \phi_{k',2q}\|_{L^\infty_t C^1} \notag\\
    &\qquad + \sum_{\substack{k'+k\neq0 \\ k,k'\in\Lambda}} 
               \lambda^{-2}_{2q} 
               \|\chi^2_{2q-1} a_{(k,2q)} a_{(k',2q)} 
                 \phi_{k,2q} \phi_{k',2q}\|_{L^\infty_t C^3} \Bigr)\notag\\
    &\lesssim \sum_{\substack{k'+k\neq0 \\ k,k'\in\Lambda}}  (C_0\lambda^{-1}_{2q} \lambda^{1/4}_{2q}  +C_0\lambda^{-2}_{2q} \lambda^{3/4}_{2q})\lesssim C_0\lambda^{-\frac{3}{4}}_{2q},\label{F11-B-1}
\end{align}
and
\begin{align*}
    \|F^{(1,1)}_{2q}\|_{L^1_t B^{1}_{\infty,\infty}} 
    &\lesssim \sum_{\substack{k'+k\neq0 \\ k,k'\in\Lambda}} 
               \lambda_{2q} 
               \|\chi^2_{2q-1} a_{(k,2q)} a_{(k',2q)} 
                 \phi_{k,2q} \phi_{k',2q}\|_{L^\infty_t C^1} \\
    &\quad + \sum_{\substack{k'+k\neq0 \\ k,k'\in\Lambda}} 
               \|\chi^2_{2q-1} a_{(k,2q)} a_{(k',2q)} 
                 \phi_{k,2q} \phi_{k',2q}\|_{L^\infty_t C^2} \\
    &\lesssim \sum_{\substack{k'+k\neq0 \\ k,k'\in\Lambda}}(C_0 \lambda^{\frac{5}{4}}_{2q}+C_0\lambda^{\frac{1}{2}}_{2q})\lesssim C_0\lambda^{\frac{5}{4}}_{2q}.
\end{align*}
The two estimates show that
\begin{align*}
 \|F^{(1,1)}_{2q}\|_{L^1_t B^{0}_{\infty,1}}\lesssim   (\|F^{(1,1)}_{2q}\|_{L^1_t B^{-1}_{\infty,\infty}})^{\frac{1}{2}}(\|F^{(1,1)}_{2q}\|_{L^1_t B^{1}_{\infty,\infty}})^{\frac{1}{2}}\lesssim C_0\lambda^{\frac{1}{4}}_{2q}.
\end{align*}
This estimate combined with \eqref{F11-B-1} yields that
\begin{align*}
 \|F^{(1,1)}_{2q}\|_{L^1_t B^{-\frac{1}{2}}_{\infty,1}}\lesssim   (\|F^{(1,1)}_{2q}\|_{L^1_t B^{-1}_{\infty,\infty}})^{\frac{1}{2}}(\|F^{(1,1)}_{2q}\|_{L^1_t B^{0}_{\infty,1}})^{\frac{1}{2}}\lesssim C_0\lambda^{-\frac{1}{4}}_{2q}.
\end{align*}

\noindent\textbf{Estimates for $F^{(1,2)}_{2q}$.}
Note that $k\perp \bar{k}$ and
\begin{align*}
&\|\big( a_k\big({\rm Id}+\tfrac{ \mathcal{R}\wpp_{2q-1}
 }{ 1000C_0}\big)\ast_{t,x}\psi_{\ell_{2q-1}}\varphi_{\ell_{2q-1}}\big)^2-a^2_k\big({\rm Id}+\tfrac{ \mathcal{R}\wpp_{2q-1}
 }{  1000C_0}\big)\|_{L^\infty_t C^0}\\
 \lesssim& (\ell_{2q-1}\|\partial_t\mathcal{R}\wpp_{2q-1}\|_{L^\infty_tC^{1}}+\ell_{2q-1}\|\partial_t\mathcal{R}\wpp_{2q-1}\|_{L^\infty_tC^{0}}\|\mathcal{R}\wpp_{2q-1}\|_{L^\infty_tC^{0}})\\
 \lesssim& \ell_{2q-1}\lambda^{3}_{2q-1},
\end{align*}
we have, by Proposition \ref{Est-fe},
\begin{align*}
    \|F^{(1,2)}_{2q}\|_{L^1_t B^{-1}_{\infty,\infty}}
    &\lesssim C_0 \bigl\| \chi^2_{2q-1} \bigl[ 
          \bigl( a_k\big({\rm Id}+\tfrac{ \mathcal{R}\wpp_{2q-1}
 }{ 1000C_0}\big)\ast_{t,x}\psi_{\ell_{2q-1}}\varphi_{\ell_{2q-1}} \bigr)^2 
          - a^2_k \bigl( \mathrm{Id} 
                + \tfrac{\mathcal{R}\wpp_{2q-1}}{1000C_0} \bigr) \bigr] \bigr\|_{L^\infty_{t} C^0} \\
    &\quad +C_0\lambda^{-1/4}_{2q}\|\chi^2_{2q-1}a_k\big({\rm Id}+\tfrac{\mathcal{R}\wpp_{2q-1}}{1000C_0}\big)\|_{L^\infty_tC^1}+C_0\lambda^{-1/2}_{2q} \|\chi^2_{2q-1}a_k\big({\rm Id}+\tfrac{\mathcal{R}\wpp_{2q-1}}{1000C_0}\big)\|_{L^\infty_tC^3}\\
    &\lesssim C_0 \ell_{2q-1} \lambda^{3}_{2q-1} 
       \lesssim C_0 \lambda^{-40}_{2q-1},
\end{align*}

\begin{align*}
    \|F^{(1,2)}_{2q}\|_{L^1_t B^{0}_{\infty,\infty}}
    &\lesssim C_0 \bigl\| \chi^2_{2q-1} \bigl[ 
          \bigl( a_k\big({\rm Id}+\tfrac{ \mathcal{R}\wpp_{2q-1}
 }{ 1000C_0}\big)\ast_{t,x}\psi_{\ell_{2q-1}}\varphi_{\ell_{2q-1}}\bigr) \bigr)^2 
          - a^2_k \bigl( \mathrm{Id} 
                + \tfrac{\mathcal{R}\wpp_{2q-1}}{1000C_0} \bigr) \bigr] \bigr\|_{L^\infty_{t} C^1} \\
    &\quad +C_0\|\chi^2_{2q-1}a_k\big({\rm Id}+\tfrac{\mathcal{R}\wpp_{2q-1}}{1000C_0}\big)\|_{L^\infty_tC^1}\\
    &\lesssim C_0 \lambda_{2q-1},
\end{align*}
and
\begin{align*}
    \|F^{(1,2)}_{2q}\|_{L^1_t B^{1}_{\infty,\infty}}
    &\lesssim C_0 \bigl\| \chi^2_{2q-1} a^2_k 
          \bigl( \mathrm{Id} + \tfrac{\mathcal{R}\wpp_{2q-1}}{1000C_0} \bigr) 
          \bigr\|_{L^\infty_{t} C^2}  + C_0 \lambda^{1/4}_{2q} 
          \bigl\| \chi^2_{2q-1} a_k 
          \bigl( \mathrm{Id} + \tfrac{\mathcal{R}\wpp_{2q-1}}{1000C_0} \bigr) 
          \bigr\|_{L^\infty_t C^1} \\
    &\lesssim C_0 \lambda^{1/4}_{2q} \lambda^2_{2q-1}.
\end{align*}
Therefore, the above three estimates show that
\begin{align*}
    \|F^{(1,2)}_{2q}\|_{L^1_t B^{0}_{\infty,1}} 
    &\lesssim \|F^{(1,2)}_{2q}\|^{1/2}_{L^1_t B^{-1}_{\infty,\infty}} 
             \|F^{(1,2)}_{2q}\|^{1/2}_{L^1_t B^{1}_{\infty,\infty}} \lesssim C_0 \lambda^{1/8}_{2q} \lambda^{-7}_{2q-1}.
\end{align*}
 and 
\begin{align*}
    \|F^{(1,2)}_{2q}\|_{L^1_t B^{-1/2}_{\infty,1}} 
    &\lesssim \|F^{(1,2)}_{2q}\|^{1/2}_{L^1_t B^{-1}_{\infty,\infty}} 
             \|F^{(1,2)}_{2q}\|^{1/2}_{L^1_t B^{0}_{\infty,\infty}} \lesssim C_0 \lambda^{-18}_{2q-1}.
\end{align*}
  Note that the other terms in $F^{(1)}_{2q}$ are of lower order, combining the estimates for $F^{(1,1)}_{2q}$ and $F^{(1,2)}_{2q}$, we obtain \eqref{est-Fq1-2}, and thus complete the proof of Proposition \ref{def-F2}.
\end{proof}
\begin{proposition}[Estimates of $F^{(2)}_{2q}$]\label{F2}
Let
\begin{align}
F^{(2)}_{2q}=&(\partial_t w^{\rm (p)}_{2q} -\Delta w^{\rm (p)}_{2q})-\Delta w^{\rm (s)}_{2q} \notag \\
&+\Div\Big( w^{\rm (p)}_{2q} \otimes( w^{\rm (s)}_{2q}+u_{2q-2}) +  (w^{\rm (s)}_{2q}+u_{2q-2})\otimes  w^{\rm (p)}_{2q} \notag \\
&+  w^{\rm (s)}_{2q} \otimes  u_{2q-2}+u_{2q-2} \otimes   w^{\rm (s)}_{2q}+ w^{\rm (s)}_{2q} \otimes  w^{\rm (s)}_{2q}\Big).\label{D-f2}
\end{align}
Then for large enough constant $C_0$, we have
\begin{align}
    \|F^{(2)}_{2q }\|_{{L}^1_t B^{ -\frac{1}{2}}_{\infty,1}}\le C^2_0\lambda^{-\frac{1}{4}}_{2q} ,\quad \|F^{(2)}_{2q }\|_{{L}^1_t B^{ 0}_{\infty,1}}\le C^2_0\lambda^{\frac{1}{4}}_{2q }.\label{est-Fq2}
\end{align}
\end{proposition}

\begin{proof}
By virtue of \eqref{dec-wp}, we rewrite 
\begin{align*}
  \partial_t  \wpp_{2q} -\Delta \wpp_{2q}= (\partial_t - \Delta)\wpm_{2q}+ (\partial_t - \Delta)\wpr_{2q}.
\end{align*}
Observing that $ (\partial_t - \Delta)(e^{-\lambda^2_{2q}t} e^{\ii\lambda_{2q}k\cdot x})=0$, one infers that
\begin{align*}
 (\partial_t - \Delta)\wpm_{2q}   
=& \lambda_{2q }e^{-\lambda^2_{2q }t}\sum_{k\in\Lambda}\big((\partial_t - \Delta)(a_{(k,2q)} \chi_{2q-1}\phi_{k,2q} ) \big)e^{\ii\lambda_{2q }k\cdot x} \bar{k}\\
&+\ii\lambda^2_{2q }e^{-\lambda^2_{2q }t}\sum_{j=1}^2\sum_{k\in\Lambda}\partial_{x_j}(a_{(k,q)}\chi_{2q-1}\phi_{k,2q})
e^{\ii\lambda_{2q }k\cdot x}k_j \bar{k} 
\end{align*}
One can  deduce  from Proposition \ref{Est-fe} that
\begin{align}
& \| (\partial_t - \Delta)\wpm_{2q} \|_{{L}^1_t B^{ -1}_{\infty,\infty}}\notag\\
\lesssim&\|\lambda_{2q}e^{-\lambda^2_{2q}t}\|_{L^1_t} \sum_{k\in\Lambda}\big(\lambda^{-1}_{2q} \|   (\partial_t - \Delta)(a_{(k,q)} \chi_{2q-1}\phi_{k,2q}) \|_{{L}^{\infty}_{t,x}}+ \lambda^{-2}_{2q} \|   (\partial_t - \Delta)(a_{(k,q)} \chi_{2q-1}\phi_{k,2q}) \|_{{L}^{\infty}_{t}C^2}\big)\notag\\
&+ \|\lambda^2_{2q}e^{-\lambda^2_{2q}t}\|_{L^1_t}\sum_{j=1}^2\sum_{k\in\Lambda}\big(\lambda^{-1}_{2q}\|  \partial_{x_j}(a_{(k,q)}\chi_{2q-1}\phi_{k,2q})\|_{{L}^{\infty}_{t,x}}+ \lambda^{-2}_{2q} \|  \partial_{x_j}(a_{(k,q)}\chi_{2q-1}\phi_{k,2q})\|_{{L}^{\infty}_{t}C^2}\big)\notag\\
\lesssim &C_0\lambda^{-3/2}_{2q}\ell^{-1}_{2q-1}+C_0\lambda^{-3/4}_{2q}\lesssim C_0\lambda^{-3/4}_{2q}.\label{wp-b-1}
\end{align}
Additionally, we have
\begin{align*}
 &\| (\partial_t - \Delta)\wpm_{2q} \|_{{L}^1_t B^{1}_{\infty,\infty}}\\
 \lesssim&\|\lambda_{2q}e^{-\lambda^2_{2q}t}\|_{L^1_t}\sum_{k\in\Lambda}\big(\|   (\partial_t - \Delta)(a_{(k,q)} \chi_{2q-1}\phi_{k,2q}) \|_{{L}^{\infty}_{t}C^1}+ \lambda_{2q} \|   (\partial_t - \Delta)(a_{(k,q)} \chi_{2q-1}\phi_{k,2q}) \|_{{L}^{\infty}_{t,x}}\big)\\
 &+\|\lambda^2_{2q}e^{-\lambda^2_{2q}t}\|_{L^1_t}\sum_{j=1}^2\sum_{k\in\Lambda}(\|  \partial_{x_j}(a_{(k,q)}\chi_{2q-1}\phi_{k,2q})\|_{{L}^{\infty}_{t}C^1}+ \lambda_{2q}\|  \partial_{x_j}(a_{(k,q)}\chi_{2q-1}\phi_{k,2q})\|_{{L}^{\infty}_{t,x}})\\
 \lesssim&C_0\lambda^{5/4}_{2q}.
\end{align*}
Consequently, one deduces 
\begin{align*}
\| (\partial_t - \Delta)\wpm_{2q} \|_{{L}^1_t B^{ 0}_{\infty,1}}
\lesssim& (\| (\partial_t - \Delta)\wpm_{2q} \|_{{L}^1_t B^{ -1}_{\infty,\infty}})^{\frac{1}{2}}(\| (\partial_t - \Delta)\wpm_{2q} \|_{{L}^1_t B^{ 1}_{\infty,\infty }})^{\frac{1}{2}}
\lesssim C_0\lambda^{\frac{1}{4}}_{2q}.
\end{align*}
Combining this estimate with \eqref{wp-b-1}, we obtain
\begin{align*}
  \| (\partial_t - \Delta)\wpm_{2q} \|_{{L}^1_t B^{ -1/2}_{\infty,1}}
  \lesssim& (\| (\partial_t - \Delta)\wpm_{2q} \|_{{L}^1_t B^{ 0}_{\infty,1}})^{\frac{1}{2}}(\| (\partial_t - \Delta)\wpm_{2q} \|_{{L}^1_t B^{ -1}_{\infty,\infty }})^{\frac{1}{2}}\lesssim C_0\lambda^{-\frac{1}{4}}_{2q}.
\end{align*}
By the same argument for estimating $\wpr_{2q}$, we  have
\begin{align}\label{est-heat-wp}
 \| (\partial_t - \Delta)\wpp_{2q} \|_{{L}^1_t B^{ -1/2}_{\infty,1}} \lesssim C_0\lambda^{-1/4}_{2q},\,\,\,   \| (\partial_t - \Delta)\wpp_{2q} \|_{{L}^1_t B^{0}_{\infty,1}} \lesssim C_0\lambda^{1/4}_{2q}.
\end{align}

Using \eqref{wp-N} and \eqref{wpq-wss3}, we deduce
\begin{align*}
    \|\Delta \wss_{2q}\|_{L^1_tB^{-1}_{\infty,\infty}}\lesssim& \| \wss_{2q}\|_{L^1_tB^{1}_{\infty,\infty}}\lesssim \|e^{-2\lambda^2_{2q}t}\|_{L^1_t}\|\wpp_{2q-1}\|_{L^\infty C^1}
\lesssim C_0\lambda^{-2}_{2q}\lambda^2_{2q-1},
\end{align*}
and 
\begin{align*}
    \|\Delta \wss_{2q}\|_{L^1_tB^{1}_{\infty,\infty}}
\lesssim&C_0\|e^{-2\lambda^2_{2q}t}\|_{L^1_t}\|\wpp_{2q-1}\|_{L^\infty_t C^3}
\lesssim C_0\lambda^{-2}_{2q}\lambda^4_{2q-1}.
\end{align*}
Hence, we have
\begin{align}\label{es-Delt-ws}
 \|\Delta \wss_{2q}\|_{L^1_tB^{0}_{\infty,1}}\lesssim C_0\lambda^{-2}_{2q}\lambda^3_{2q-1},\quad  \|\Delta \wss_{2q}\|_{L^1_tB^{-\frac{1}{2}}_{\infty,1}}\lesssim C_0\lambda^{-2}_{2q}\lambda^{\frac{5}{2}}_{2q-1}.
\end{align}
We infer from \eqref{u-Linfty}, Proposition \ref{Est-wpws} and \eqref{D-f2} that
\begin{align*}
&\Big \|F^{(2)}_{2q}-\big((\partial_t  \wpp_{2q} -\Delta \wpp_{2q})-\Delta \wss_{2q} \big)\Big\|_{L^1_tB^{-1}_{\infty,\infty}}\\
\lesssim&\|\wpp_{2q}\|_{L^\infty_{t,x}}(\|(\wss_{2q},u_{2q-2})\|_{L^1_{t}L^\infty}+\|\wss_{2q}\|_{L^1_{t}L^\infty}(\|(\wss_{2q},u_{2q-2})\|_{L^\infty_{t,x}})\\
\lesssim&C^2_0\lambda^{-1}_{2q}\lambda_{2q-1}
\end{align*}
and 
\begin{align*}
&\Big \|F^{(2)}_{2q}-\big((\partial_t  \wpp_{2q} -\Delta \wpp_{2q})-\Delta \wss_{2q} \big)\Big\|_{L^1_tB^{1}_{\infty,\infty}}\\
\lesssim&\|\wpp_{2q}\|_{L^1_{t}C^2}\|(\wss_{2q},u_{2q-2})\|_{L^\infty_{t,x}}+\|\wpp_{2q}\|_{L^\infty_{t,x}}\|(\wss_{2q},u_{2q-2})\|_{L^1_{t}C^2}\\
&+\|\wss_{2q}\|_{L^1_{t}C^2}\|(\wss_{2q},u_{2q-2})\|_{L^\infty_{t,x}}+\|\wss_{2q}\|_{L^\infty_{t,x}}\|(\wss_{2q},u_{2q-2})\|_{L^1_{t}C^2}\\
\lesssim&C^2_0\lambda_{2q}\lambda_{2q-1}.
\end{align*}
Hence, it follows that
\begin{align}
  \Big \|F^{(2)}_{2q}-\big((\partial_t  \wpp_{2q} -\Delta \wpp_{2q})-\Delta \wss_{2q} \big)\Big\|_{L^1_tB^{0}_{\infty,1}}  \lesssim C^2_0\lambda_{2q-1},
\end{align}
and 
\begin{align}\label{Es-F2-other}
  \Big \|F^{(2)}_{2q}-\big((\partial_t  \wpp_{2q} -\Delta \wpp_{2q})-\Delta \wss_{2q} \big)\Big\|_{L^1_tB^{-\frac{1}{2}}_{\infty,1}}  \lesssim C^2_0\lambda^{-1/2}_{2q}\lambda_{2q-1}.
\end{align}
Collecting  \eqref{est-heat-wp}--\eqref{Es-F2-other} together, we complete the proof of Proposition \ref{F2}.
\end{proof}
\subsection{Construction of $\wtq$ and completion of iteration}\label{Con-wns}
Let $\wtq$ solve the following Cauchy problem
 \begin{equation}
\left\{ \begin{alignedat}{-1}
&\del_t \wtq-\Delta \wtq+\Div \big(\wtq\otimes \wtq\big)+\Div\big(\wtq\otimes(u_{2q-2}+\wpp_{2q}+\wss_{2q})\big) \\
&+\Div((u_{2q-2}+\wpp_{2q}+\wss_{2q})\otimes \wtq)  +\nabla p^{(\textup{ns})}_{2q }  = -F_{2q },
\\
 & \Div \wtq = 0,
  \\
  & \wtq |_{t=0}=  0 ,
\end{alignedat}\right.
 \label{e:wt}\tag{FNS}
\end{equation}
 where $F_{2q }:=F^{(1)}_{2q }+F^{(2)}_{2q }$ stems from Propositions \ref{def-F2} and \ref{F2}.

 \begin{proposition}[Estimates for $\wtq$]\label{wtq-H3}
There  exists a global-in-time solution $\wtq \in C^\infty_{t,x}(\R^{+}\times\TTT^2)$ of the system \eqref{e:wt}
 satisfying
\begin{align}
\|\wtq\|_{\widetilde L^{\infty}_tB^{-1/2}_{\infty,1}\cap L^{1}_tB^{3/2}_{\infty,1}} \le \lambda^{-10}_{2q-1}, \label{wtq1}
\end{align}
Moreover, we have
\begin{align}
\|\wtq\|_{\widetilde L^{\infty}_tB^{0}_{\infty,1}\cap L^{1}_tB^{2}_{\infty,1}} \le \lambda^{1/2}_{2q}.\label
{wtq2}
\end{align}
\end{proposition}

\begin{proof}
Applying the Leray projector 
$\mathbb{P}_{H} = \mathrm{Id} - \Delta^{-1}\nabla\Div$ to equation \eqref{e:wt}
, we obtain
\begin{align*}
&
\partial_t \wtq - \Delta \wtq + (\wss_{2q} + \wpp_{2q} + u_{2q-2}) \cdot \nabla \wtq \\
=& 
\Delta^{-1}\nabla\Div\big((\wss_{2q} + \wpp_{2q} + u_{2q-2}) \cdot \nabla \wtq\big) \\
&
\quad - \mathbb{P}_{H}\big(\wtq \cdot \nabla(\wtq + \wss_{2q} + \wpp_{2q} + u_{2q-2}) + F_{2q}\big
).
\end{align*}
Using Proposition~
\ref{T-D}
, we deduce that
\begin{align}
&
\|\wtq\|_{\widetilde L^{\infty}_TB^{-1/2}_{\infty,1}\cap L^{1}_TB^{3/2}_{\infty,1}} \notag\\
\le& C\exp\Big(C\| \wss_{2q} + \wpp_{2q} + u_{2q-2} \|_{L^1_TC^1}\Big) \Big( \|\Delta^{-1}\nabla\Div\big((\wss_{2q} + \wpp_{2q} + u_{2q-2}) \cdot \nabla \wtq\big) \|_{L^{1}_tB^{-1/2}_{\infty,1}} \notag\\
&+ 
\|\mathbb{P}_{H}\big(\wtq \cdot \nabla(\wtq + \wss_{2q} + \wpp_{2q} + u_{2q-2})\big) \|_{L^{1}_tB^{-1/2}_{\infty
,1}}
+ 
\|\mathbb{P}_{H} F_{2q} \|_{L^{1}_tB^{-1/2}_{\infty,1}} \Big). \label
{E-wns}
\end{align}
Since 
$\Div\wss_{2q} = \Div\wpp_{2q} = \Div\wtq = \Div u_{2q-2} = 0$
, we have
\begin{align*}
&
\|\Delta^{-1}\nabla\Div\big((\wss_{2q} + \wpp_{2q} + u_{2q-2}) \cdot \nabla \wtq\big) \|_{B^{-1/2}_{\infty,1}} \\
&+ 
\|\mathbb{P}_{H}\big(\wtq \cdot \nabla(\wtq + \wss_{2q} + \wpp_{2q} + u_{2q-2})\big) \|_{B^{-1/2}_{\infty,1}} \\
\lesssim& \|(\wtq, \wss_{2q}, \wpp_{2q}, u_{2q-2})\|_{C^1} \|\wtq\|_{B^{-1/2}_{\infty
,1}}.
\end{align*}
Putting these estimates into 
\eqref{E-wns}
 yields
\begin{align}
&
\|\wtq\|_{\widetilde L^{\infty}_TB^{-1/2}_{\infty,1}\cap L^{1}_TB^{3/2}_{\infty,1}} \notag\\
\le& C\exp\big(CC_0(q+1)\big) \Big( \int_0^t \|(\wtq, \wss_{2q}, \wpp_{2q}, u_{2q-2})\|_{C^1} \|\wtq\|_{B^{-1/2}_{\infty,1}} \, \dd s + C_0\lambda^{-15}_{2q-1} \Big). \label
{E-wns2}
\end{align}
Let
\begin{align}
T^{
\star} := \sup\Big\{ T \in [0,\infty) \,\big|\, \|\wtq\|_{\widetilde L^{\infty}_TB^{-1/2}_{\infty,1}\cap L^{1}_TB^{3/2}_{\infty,1}} \leq \lambda^{-10}_{2q-1} \Big\}. \label
{def-T}
\end{align}
Then for any 
$T\in [0,T^{\star}]$
, we have
\begin{align*}
&
\|\wtq\|_{\widetilde L^{\infty}_TB^{-1/2}_{\infty,1}\cap L^{1}_TB^{3/2}_{\infty,1}} \\
\le& 2C\exp\big(CC_0(q+1)\big) \Big( \int_0^T  \| \wss_{2q} + \wpp_{2q} + u_{2q-2} \|_{C^1} \|\wtq\|_{B^{-1/2}_{\infty,1}} \, \dd s + C_0\lambda^{-15}_{2q-1} \Big
).
\end{align*}
Using Gr{o}nwall's inequality, we obtain
\begin{align}
&\|\wtq\|_{\widetilde L^{\infty}_TB^{-1/2}_{\infty,1}\cap L^{1}_TB^{3/2}_{\infty
,1}}\\
\le& 2CC_0\exp\big(CC_0(q+1)\big) \lambda^{-15}_{2q-1} \exp\Big(  2C\exp\big(CC_0(q+1)\big)\int_0^T \| \wss_{2q} + \wpp_{2q} + u_{2q-2} \|_{C^1}\, \dd s \Big) \notag \\
\le& \exp(2\exp(2CC_0(q+1))) \lambda^{-15}_{2q-1}, \label
{E-wns3}
\end{align}
where we have used the estimates 
\eqref{u-L1C1} and \eqref{wpq-wss2}
 to bound
\begin{align*}
\| \wss_{2q} + \wpp_{2q} + u_{2q-2} \|_{L^1_TC^1} \le
 C_0(q+1).
\end{align*}
By choosing $b$ large enough such that
$b\ge 10e^{10CC_0}$, one infers from \eqref{E-wns3}
 that
\begin{align*}
\|\wtq\|_{\widetilde L^{\infty}_TB^{-1/2}_{\infty,1}\cap L^{1}_TB^{3/2}_{\infty,1}} \le \lambda_{2q-1} \lambda^{-15}_{2q-1} = \lambda
^{-14}_{2q-1}.
\end{align*}
 From the definition \eqref{def-T} and a continuity argument, we conclude that $T^{\star} = \infty$. Moreover, one can deduce the estimate \eqref{wtq2} by a similar argument. Thus the proof of Proposition~
\ref{wtq-H3}
 is completed.
\end{proof}

 With the aid of \eqref{def-wpq-1}, Propositions \ref{Est-wpws} and  \ref{cedu}, one immediately deduces that $\wpp_{2q}$, $\wss_{2q}$ and $\wtq$ satisfy \eqref{def-R}--\eqref{e-ns}.
  Let
   \begin{align*}
       w_{2q}=\wpp_{2q}+\wss_{2q}+\wtq,
   \end{align*}
and
   \begin{align*}
    u_{2q}=u_{2q-2}+w_{2q}.
\end{align*}
We immediately show \eqref{u-decomposition} for $j=m=2q$. Next, we verify that $u_{2q}$ satisfies the estimates stated in Proposition \ref{iteration}. By \eqref{u-Linfty} and \eqref{wm-infty}, we have
\begin{align*}
   \|u_{2q}\|_{L^\infty_{t,x}}\le& \|u_{2q-2}\|_{L^\infty_{t,x}}+\|w_{2q}\|_{L^\infty_{t,x}}
   \le C_0\lambda_{2q-2}+\frac{1}{2}C_0\lambda_{2q}\le C_0\lambda_{2q},
\end{align*}
and
\begin{align*}
   \|u_{2q}\|_{L^1_{t}C^2}\le& \|u_{2q-2}\|_{L^1_{t}C^2}+\|w_{2q}\|_{{L^1_{t}C^2}}
   \le C_0\lambda_{2q-2}+\frac{1}{2}C_0\lambda_{2q}\le C_0\lambda_{2q},
\end{align*}
which gives \eqref{u-Linfty} for $j=m=2q$.
It follows from \eqref{u-L1C1} and \eqref{wm-infty} that
\begin{align*}
   \|u_{2q}\|_{L^1_{t}C^1}\le& \|u_{2q-2}\|_{L^1_{t}C^1}+\|w_{2q}\|_{L^1_{t}C^1}
   \le C_0(q-1)+\frac{1}{2}C_0\le C_0q.
\end{align*}
This estimate shows \eqref{u-L1C1} for $j=m=2q$. Propositions  \ref{BMO} and \ref{wtq-H3}  together show that
\begin{align*}
 \|u_{2q}\|_{L^\infty \BMO^{-1}}\le \frac{1}{2}C_0+ 2\lambda^{-10}_1\le C_0,
\end{align*}
which gives \eqref{e-Um-BMO}. Based on  the equality \eqref{def-ws} and the system \eqref{e:wt}, we  have  \eqref{ini-wm}. The fact $$\widetilde L^{\infty}_tB^{-1/2}_{\infty,1}(\TTT^2)\cap L^{1}_tB^{3/2}_{\infty,1}(\TTT^2)\hookrightarrow L^{2}_tL^p(\TTT^2) \cap L^{1}_t  W^{1,p}(\TTT^2),\,\,\,\forall p\ge 1,$$ 
together with \eqref{wp-L} and \eqref{e-ns} yields \eqref{wm-L}.  Hence, we complete the proof of Proposition \ref{iteration} for the case where $m$ is even. For the case that $m$ is odd, we have $m=2q+1$ for some $q \ge 1$. We construct $\wpp_{2q+1}, \wss_{2q+1}$ and $ w^{(\text{ns})}_{2q+1}$ by replacing the parameter $2q$ in definitions \eqref{def-wpq-1}, \eqref{def-ws} and in the system \eqref{e:wt} with $2q+1$. Let $w_{2q+1}=\wpp_{2q+1}+\wss_{2q+1}+w^{\text{(ns)}}_{2q+1}$ and $u_{2q+1}=u_{2q-1}+w_{2q+1}.$ Following the same argument as for the even case, Proposition \ref{iteration} is  established also for odd $m$, which finishes the proof of Proposition \ref{iteration}.

 \section{Proof of Theorem \ref{t:main}}\label{p-t}
By virtue of Proposition \ref{iteration}, we obtain two sequences of smooth solutions $\{ u_{2q} \}_{q\ge 1}$ and $\{ u_{2q+1} \}_{q\ge 0}$ to the Navier--Stokes equations \eqref{NS}. The identity $u_{m}-u_{m-2}=w_m$, together with estimate \eqref{wm-L}, implies that both $\{ u_{2q} \}_{q\ge 1}$ and $\{ u_{2q+1} \}_{q\ge 1}$ are Cauchy sequences in the space $L^{2}_tL^r\cap L^{1}_tW^{1,r}$ for every $1\le r<\infty$. Consequently, there exist two functions  
$u^{\mathrm{even}}, u^{\mathrm{odd}}\in L^{2}_tL^r\cap L^{1}_tW^{1,r}$ such that  
\[
u_{2q}\longrightarrow u^{\mathrm{even}},\qquad 
u_{2q+1}\longrightarrow u^{\mathrm{odd}},\qquad 
\text{strongly in } L^{2}_tL^r\cap L^{1}_tW^{1,r}\ \text{ as } q\to\infty .
\]

On the other hand, from \eqref{e-Um-BMO} we infer that
\begin{equation} \label{eq:weakstar_converge}
u_{2q} \weakstar u^{\mathrm{even}},\qquad 
u_{2q+1} \weakstar u^{\mathrm{odd}},\qquad 
\text{weakly-* in } L^\infty_t {\rm BMO}^{-1} \text{ as } q \to \infty .
\end{equation}
Hence $u^{\mathrm{even}}, u^{\mathrm{odd}}\in L^\infty_t \BMO^{-1}$, and 
\begin{align*}
u^{\mathrm{even}}=\sum_{l=1}^{\infty} \bigl(\wpp_{2l}+\wss_{2l}+w^{\mathrm{(ns)}}_{2l}\bigr),
\,\,\,
u^{\mathrm{odd}} =\sum_{l=0}^{\infty} \bigl(\wpp_{2l+1}+\wss_{2l+1}+w^{\mathrm{(ns)}}_{2l+1}\bigr).
\end{align*}
From \eqref{ini-wm}, we deduce  
$$u_{2q}(0,x)=u_{2q+1}(0,x)-\wpp_{2q+1}(0,x).$$  
From the proof of Proposition \ref{cedu}, one easily verifies  that $w^{(\rm p)}_{2q+1} \to 0$ strongly  in $W^{-1,p}(\mathbb{T}^2)$ for every $p<\infty$, which yields  
\[
u^{\mathrm{even}}(0,x)=u^{\mathrm{odd}}(0,x)\qquad\text{in the sense of distributions}.
\]

We now turn to the non‑uniqueness of $u^{\mathrm{even}}$ and $u^{\mathrm{odd}}$.  
Recalling that  
\[
\wpp_{1}=u_1=\lambda_1 \sin(\lambda_1x_1)e^{-\lambda^2_1t}e_2,.
\]
we obtain, at $t=\lambda_1^{-2}$,
\begin{align*}
\bigl\|\wpp_{1}(\lambda_1^{-2},\cdot)\bigr\|_{B^{-1}_{\infty,\infty}}
&=\|\lambda_1\sin(\lambda_1x_1)\|_{B^{-1}_{\infty,\infty}}e^{-1}=:M_0.
\end{align*}
Choosing $C_0$ large enough such that $M_0\le C^{1/2}_0$, one has
\begin{align*}
&\sum_{q\ge 1}
\bigl\|(\wpp_{2q},\wss_{2q},\wtq, \wpp_{2q+1}, \wss_{2q+1}, 
      w^{\mathrm{(ns)}}_{2q+1})(\lambda_1^{-2},\cdot)\bigr\|_{B^{-1}_{\infty,\infty}} \\
\lesssim&\sum_{q\ge 1}\bigl(C_0^{2}e^{-\lambda^2_{2q}\lambda_1^{-2}}
          +\lambda_{2q-1}^{-10}\bigr)\lesssim\lambda_1^{-8},
\end{align*}
where the last inequality holds for large enough $a$. Therefore, we have 
\begin{align*}
&\bigl\| (u^{\mathrm{odd}}-u^{\mathrm{even}})(\lambda_1^{-2},\cdot)\bigr\|_{B^{-1}_{\infty,\infty}}\\
\ge& \bigl\|\wpp_{1}(\lambda_1^{-2},\cdot)\bigr\|_{B^{-1}_{\infty,\infty}} -\sum_{q\ge 1}6\bigl\|(\wpp_{2q},\wtq, \wss_{2q}, \wpp_{2q+1}, 
        \wss_{2q+1}, w^{\mathrm{(ns)}}_{2q+1})(\lambda_1^{-2},\cdot)
        \bigr\|_{B^{-1}_{\infty,\infty}} \\
\ge& M_0 -C\lambda_1^{-8}
      \ge \frac{M_0}{2} ,
\end{align*}
as long as  $a$ is sufficiently large. From the constructions of \( \wpp_m \) and \( \wss_m \), one can deduce that
 \[
 \sum_{m=1}^\infty (\wpp_{2m} + \wss_{2m}), \quad \sum_{m=0}^\infty (\wpp_{2m+1} + \wss_{2m+1}) \in C_{t,x}^{\infty}(\mathbb{R}^+ \times \mathbb{T}^2).
\]
 Moreover, it follows  that, for some \( \alpha > 0 \),
\[\sum_{m=1}^\infty w_{2m}^{(\text{ns})}, \quad \sum_{m=0}^\infty w_{2m+1}^{(\text{ns})} \in C^0((0,\infty);C^{1,\alpha}(\mathbb{T}^2) ), \]
where we use the definition of $F_m$.
Noting that \( u^{\text{odd}} \) and \( u^{\text{even}} \) satisfy the system \eqref{NS},  we conclude that both \( u^{\text{odd}} \) and \( u^{\text{even}} \) are smooth for \( t > 0 \). Readers can refer to \cite{CP} for the detailed proof. Since $\BMO^{-1}(\TTT^2)\hookrightarrow B^{-1}_{\infty,\infty}(\TTT^2)$, we show $u^{\mathrm{odd}}\neq u^{\mathrm{even}}$ in $L^\infty_t\BMO^{-1}$. Therefore, we complete the proof of Theorem ~\ref{t:main}.

\appendix

\section{}\label{App}
In this section, we compile several useful tools {including geometric lemma,  the definitions of mollifiers and Chemin-Lerner
space,  the local well-posedness of  the tranport-diffusion system.}
\begin{definition}[Mollifiers]\label{e:defn-mollifier-t}Let nonnegative functions $\varphi(t)\in C^\infty_c(-1,0)$ and $\psi(x)\in C^\infty_c(B_1(0))$ be standard mollifying kernels such that $\int_{\R}\varphi(t)\dd t=\int_{\R^2}\psi(x)\dd x=1$.
For each $\epsilon>0$, we define  two sequences of  mollifiers as follows:
\begin{align*}
     \varphi_{\epsilon}(t)
        \coloneq \frac1{\epsilon} \varphi\left(\frac t\epsilon\right), \quad\psi_\epsilon(x)
            \coloneq \frac1{\epsilon^2} \psi\left(\frac{x}\epsilon\right).
\end{align*}
\end{definition}
The choice of such compact support in time for $\varphi(t)$ is due to the fact that  the convolved object in the convolution \eqref{def-akq} is defined only for $t\ge 0$.
\begin{lemma}[Stationary flows in 2D \cite{CDS}]\label{Betrimi}Let $\mathbb{S}^1$ be the unit circle and $\mathbb{Q}$  the set of rational numbers.
Given $\Lambda\subseteq\mathbb{S}^1\cap\mathbb{Q}^2$ such that $-\Lambda=\Lambda$. Then for any $b_k\in\mathbb{R}$ with ${b_{k}}=b_{-k}$, the vector field
$$ W(\xi)=\sum_{k\in\Lambda}b_k\frac{\ii \bar{k}}{|k|}e^{ \ii k\cdot \xi}, k\perp\bar{k}, $$
is real-valued, divergence-free and satisfies
$$\Div_{\xi}(W\otimes W)= \frac{1}{2}\nabla_{\xi}(|W|^2-\Big|\sum_{k\in\Lambda}\frac{b_k}{|k|}e^{\ii k\cdot\xi}\Big|^2)$$
and \begin{align}\label{b-g}
W\otimes W=&\sum_{j,k\in\Lambda,~j+k\neq0}-b_j b_k e^{\ii (j+k)\cdot\xi}\frac{  \bar{j}}{|j|}\otimes \frac{  \bar{k}}{|k|}+\sum_{j,k\in\Lambda,~j+k=0} b_k^2\frac{  \bar{k}}{|k|}\otimes \frac{  \bar{k}}{|k|}\notag\\
=&\sum_{j,k\in\Lambda,~j+k\neq0}-\frac{b_k}{|k|}\frac{b_j}{|j|}  e^{\ii (j+k)\cdot\xi} {  \bar{j}} \otimes  { \bar{k}} +\sum_{j,k\in\Lambda,~j+k=0} \Big(\frac{b_k}{|k|}\Big)^2   \bar{k}\otimes\bar{k}.
%=&\sum_{j,k\in\Lambda,~j+k\neq0}-a_j a_k{j}e^{\ii (j+k)\xi}\frac{  j^{\perp}}{|j|}\otimes \frac{  k^{\perp}}{|k|}+\sum_{k\in\Lambda} a^2_k(Id- \frac{  k }{|k|}\otimes \frac{  k }{|k|}).
\end{align}

%where $\frac{  k^{\perp}}{|k|}\otimes \frac{  k^{\perp}}{|k|}=Id-\frac{  k }{|k|}\otimes \frac{  k }{|k|}$.
\end{lemma}
\begin{lemma}[{Geometric Lemma} \cite{CDS}]\label{first S}Let $\epsilon>0$ and $B_{\sigma}(0)$ denote the ball of radius $\sigma$ centered at $\rm Id$ in the space of $2\times2$ symmetric matrices. There exist $\Lambda\subseteq \mathbb{S}^1\cap\mathbb{Q}^2 $ 
that consists of vectors $k$ with associated orthonormal basis $(k, \bar{k} )$  and smooth function $a_{k}:B_{10^{-3}}(\rm Id)\rightarrow\mathbb{R}$ such that for each $R\in B_{10^{-3}}(\rm Id)$, we have the following identity:
$$R=\sum_{k\in\Lambda}a^2_{k}(R) \bar{k}\otimes\bar{k}.$$
\end{lemma}

We give the definitions of Besov spaces and mixed time-spatial Besov spaces (the so-called Chemin-Lerner spaces.)
\begin{definition}[\cite{MNY, Sch, WZ}]Let $s\in\mathbb{R}$ and $1\le p,q\le\infty$.  The nonhomogeneous Besov space $ B^s_{p,q}(\TTT^2)$ consists of all $\TTT^2$-periodic distributions $u\in \mathcal{D}'(\TTT^2)$  such that
\begin{equation}\nonumber
\|u\|_{ B^s_{p,q}(\TTT^2)}\overset{\text{def}}{=}\Big{\|}\left(2^{js}\|\Delta_j u\|_{L^p(\TTT^2)}\right)_{j\in\ZZ}\Big{\|}_{\ell^q(\ZZ)}<\infty.
\end{equation}
\end{definition}

\begin{definition}Let $T>0$, $s\in\R$ and $1\le r,p,q\le\infty$. The mixed  time-spatial Besov spaces ${L}^r_TB^s_{p,q}$ consists of all $u\in\mathcal{D}'$ satisfying
\begin{align*}
\|u\|_{\widetilde{L}^r_TB^s_{p,q}(\TTT^2)}\overset{\text{def}}{=}\Big{\|}(2^{js}\|\Delta_j u\|_{L^r([0,T];L^p(\TTT^2))})_{j\in\ZZ}\Big{\|}_{\ell^q(\ZZ)}<\infty.
\end{align*}
\end{definition}

\begin{lemma}[\cite{BCD11}]\label{T-D}
Let \(1  \leq p  \leq \infty\), $-1-2\min\{\tfrac{1}{p}, \tfrac{1}{p'}\}< s<1$. Consider the transport-diffusion equation
\begin{align}\label{T-D-E}
\partial_t u-\Delta u+v\cdot \nabla u=g,\qquad
  u(0,x)=u_0(x),
  \end{align}
where $u_0$ and $g$ stand for mean-zero given initial data and external force, respectively. And $v$ is  divergence-free vector field. There exists a constant \(C\) which depends only on  \(s\) and $p$ and is such that for any smooth solution \(u\) of \eqref{T-D-E} and \(1 \leq \rho_1 \leq \rho \leq \infty\), we have
\[
\|u\|_{\tilde{L}_T^{\rho}B^{s+\frac{2}{\rho}}_{p,1}(\TTT^2)} \leq C e^{C V_{p}(T)} 
\left(  \|u_0\|_{B_{p,1}^s(\TTT^2)} +\|g\|_{\tilde{L}_T^{p_1}(B_{p,1}^{s-2+\frac{2}{\rho_1}}(\TTT^2))} \right),
\]
where 
\begin{align*}
 V_{p}(T)=\int_0^T\|\nabla v(s)\|_{ L^\infty(\TTT^2)}\dd s.
\end{align*}
\end{lemma}

\begin{proposition}\label{Est-fe}For $k\in\mathbb{Q}^2\setminus\{0\}$, we have
\begin{align*}
\|fe^{\ii\lambda k\cdot x}\|_{L^\infty_tB^{-1}_{\infty,\infty}(\TTT^2)}\lesssim_{|k|}\lambda^{-1}\|f\|_{L^\infty_tL^\infty(\TTT^2)}+\lambda^{-2}\|f\|_{L^\infty_tC^2(\TTT^2)}.
\end{align*}
\end{proposition}
\begin{proof}
Since $e^{\ii\lambda k\cdot x}=\frac{\Div(e^{\ii\lambda k\cdot x}k)}{\ii\lambda |k|}$, we have
\begin{align*}
fe^{\ii\lambda k\cdot x}=&\frac{1}{\ii\lambda |k|}\Div(fe^{\ii\lambda k\cdot x}k)-\frac{1}{\ii\lambda |k|}e^{\ii\lambda k\cdot x}k\cdot\nabla f\\
=&\frac{1}{\ii\lambda |k|}\Div(fe^{\ii\lambda k\cdot x}k)-\frac{1}{(\ii\lambda |k|)^2}\Big(\Div(k\cdot\nabla fe^{\ii\lambda k\cdot x}k)-e^{\ii\lambda k\cdot x}k\cdot\nabla (k\cdot\nabla f)\Big).
\end{align*}
Therefore, we conclude that
\begin{align*}
&\|fe^{\ii\lambda k\cdot x}\|_{L^\infty_tB^{-1}_{\infty,\infty}(\TTT^2)}\\
\lesssim&\lambda^{-1}\|fe^{\ii\lambda k\cdot x}\|_{B^0_{\infty,\infty}(\TTT^2)}+\lambda^{-2}(\|(\nabla f)e^{2\pi\ii\lambda k\cdot x}\|_{B^0_{\infty,\infty}(\TTT^2)}+\|(\nabla^2 f)e^{\ii\lambda k\cdot x}\|_{B^{-1}_{\infty,\infty}(\TTT^2)})\\
\lesssim&\lambda^{-1}\|f\|_{L^\infty_tL^\infty(\TTT^2)}+\lambda^{-2}\|f\|_{L^\infty_tC^2(\TTT^2)}.
\end{align*}
\end{proof}
%%%%%%%%%%%%%%%%%%%%%%%%%%%%%%%%%%%%%%%%%%%%%%%%%%%%%%%%%%%%%%%%%%%%%%%%%%%%%%%%%%%%%%%%%%%%%%%%%%%%%%%%%%%%%%%%%%%%%%%%%%%%%%

\section*{Acknowledgement}
We thank Professor Palasek for a helpful exchange, through which he kindly informed us of the independent and simultaneous work of himself, Cheskidov, and Dai \cite{CDP}. We are pleased to note that both groups arrived at a similar conclusion through different methods. We also thank Professors Cheskidov and Dai for their valuable comments on this manuscript.
This work was supported by the National Key Research and Development Program of China    (Grant number No. 2022YFA1005700) and  the National Natural Science Foundation of China (Grant numbers No.12371095, No. 12401277 and No. 
12501304).

\end{document}